\documentclass[12pt,letterpaper]{amsart}
\makeatletter
\@namedef{subjclassname@2020}{%
\textup{2020} Mathematics Subject Classification}
\makeatother
\usepackage{geometry,xcolor}
\usepackage{setspace}
\usepackage{mathpazo} 
\usepackage[colorlinks,allcolors=blue]{hyperref} 
\usepackage{xpatch,multirow}
\xpatchcmd{\proof}{\itshape}{\bfseries}{}{}
\newtheorem{theorem}{Theorem}

\newtheorem{corollary}{Corollary}
\newtheorem{lemma}{Lemma}

\theoremstyle{remark}
\newtheorem{remark}{Remark}

\usepackage{soul}
\allowdisplaybreaks

\title{The curvature estimation of the complete K\"ahler-Einstein metrics on the disk bundles}
\author{Yihong Hao}
\address[Yihong Hao]{Department of Mathematics, Northwest University, Xi'an \rm{710127}, China}
\email{haoyihong@nwu.edu.cn}

\author{Mingming Chen}
\address[Mingming Chen]{School of Mathematics and Statistics, Henan Normal University, Xinxiang 453007, China}
\email{chenmingming@htu.edu.cn}

\author{An Wang}
\address[An Wang]{School of Mathematical Sciences, Capital Normal University, Beijing \rm{100048}, China}
\email{wangan@cnu.edu.cn}

\author{Ben Zhang}
\address[Ben Zhang]{Department of Mathematics, Northwest University, Xi'an \rm{710127}, China}
\email{alex\_zhangben@163.com}

\date{\today}
\begin{document}
	
\begin{abstract}
In this paper, we study  some geometric properties of the disk bundles  over the complete K\"ahler manifolds.
When the base space is the complete K\"ahler-Einstein manifold, we
compute the holomorphic sectional curvature, Riemannian sectional curvature of the complete K\"ahler-Einstein metric on the disk bundle.
Moreover, we present a parametric criteria for whether the manifold admits negative pinched property.
When the base space is a bounded pseudoconvex domain equipped with its complete K\"ahler metric,
the corresponding ball  bundle in the $k$-direct sum of line bundles is a bounded pseudoconvex Hartogs domain with fiber dimension $k$ . If
the base domain is simply-connected and $k\geq2$, we prove that the Bergman metric on such a Hartogs domain is K\"ahler-Einstein if and only if
the domain is biholomorphically equivalent to a unit ball.
\end{abstract}
\maketitle

\section{Introduction}
Let $(M, g_{M})$ be a  K\"ahler manifold, and  $\pi:L\rightarrow M$ be a Hermitian line bundle equipped with a
Hermitian fiber metric $h$ whose curvature $\Theta$ satisfying $-\frac{\sqrt{-1}}{2}\Theta=\partial\overline{\partial}\log h=l\omega_{M}$ for $l\in \mathbb{R}$. Let $f$ be a smooth real-valued  function on $[0,+\infty)$.
Then the  $(1, 1)$-form
\begin{equation}\label{equ:CAn}
\omega_{f}=\pi^{*}(\omega_M)+\frac{\sqrt{-1}}{2}\partial\bar\partial f(|\zeta|_{h}^{2})
\end{equation}
is well defined on the total space of  $L$. It is called Calabi ansatz.
The disk bundle is defined by
 \begin{equation}
 D(L) := \{\zeta  \in L : |\zeta|_{h}^{2} < 1\},
 \end{equation}
where $|\zeta|_{h}$ is the norm of $\zeta$ with respect to the metric $h$.
Under the assumption that $g_{M}$ is a K\"ahler-Einstein metric with  Ricci curvature $\lambda_{0}$,
 Calabi \cite{Calabi1979}  reduced the problem of the existence of the K\"ahler-Einstein
metrics to a problem of solving ordinary differential equations.
 He proved that the form \eqref{equ:CAn} induces  a complete K\"ahler-Einstein metric with Einstein constant $\lambda$ on $ D(L)$  if and only if:
(a) $M$ is complete with respect to  $g_{M}$;
(b)  $\lambda=\lambda_{0}-l$;
(c) $l>0$ and $\lambda\leq0$.

 In \cite{Bland1986}, Bland described the complete K\"ahler-Einstein metric for the domain
 \begin{equation}\label{equ:Dp}
 D_{p}=\{(z,\zeta)\in \mathbb{B}^{n}\times\mathbb{C}:\|z\|^{2}+|\zeta|^{2p}<1, p>0\}.
 \end{equation}
 It can be seen as a unit disk bundle in the trivial line bundle over the ball $\mathbb{B}^{n}$.
 The Hermitian fiber metric $h$ is  $ (1-\|z\|^{2})^{-\frac{1}{p}}$ .
  His main technique is to exploit the
 noncompact automorphism group
 and the automorphism invariance of the K\"ahler-Einstein metric to reduce the Monge-Amp\`ere equation for the metric to an ordinary differential equation. He also showed that the Riemannian sectional curvature for the domain is bounded between negative constants if $p\geq1$.
After that,  Wang, Yin, Zhang, Roos \cite{Wang2006}  extended these results to a class of Hartogs domains over irreducible bounded symmetric domains.
  Recall Calabi's conditions (b) and (c), we know  that  the values  of the parameter $l$ and  the curvature $k$  is not the key by scaling.
 Without loss of generality, one can suppose that $l=-1$, $\lambda=-(n+1)$. Under this assumption,
  Ebenfelt, Xiao and Xu \cite{Ebenfelt2023} introduced a new
ansatz to study the existence of the  complete K\"ahler-Einstein metric  on the disk bundle over a complete K\"ahler manifold with constant Ricci eigenvalues. Their  sufficient condition is that all Ricci eigenvalues  must be strictly less than one.

Our first main work of this paper is to compute the holomorphic sectional curvature and Riemannian sectional curvature for the complete K\"ahler-Einstein metric on the disk bundle \eqref{equ:CAn}.
It is motivated by Bland's  curvature estimation for the complete K\"ahler-Einstein metric on the Thullen domain \eqref{equ:Dp}.
When the base space of the disk bundle \eqref{equ:CAn} is a complete K\"ahler-Einstein manifold, we study  the negatively (holomorphically) pinched properties of the  complete K\"ahler-Einstein metric on the disk bundle.
As a direct application of Wu and Yau's result in \cite{Wu2020}, we can obtain additional complex manifolds on which the complete K\"ahler-Einstein metric, the Kobayashi metric and the Bergman metric are equivalent.

Another motivation for this paper comes from  Cheng's conjecture in \cite{Cheng1979}:
 if the Bergman metric of a strictly
pseudoconvex domain in $\mathbb{C}^{n}$ is K\"ahler-Einstein, then the domain is biholomorphic to
the complex unit ball. Replace the strictly pseudoconvex domain by a pseudoconvex domain in $\mathbb{C}^{n}$. It is well known that
the Bergman metric and the complete K\"ahler-Einstein metric are coincide on any bounded homogeneous domain.
It was also asked by Yau (\cite{Yau1982}, pg. 679), whether this happens if and only if the bounded domain is homogeneous.
 Cheng's conjecture was confirmed by the combined work of
Fu-Wong\cite{Fu1997} and Nemirovski-Shafikov \cite{Nem2006} in dimension two. Finally, it was proved by   Huang and Xiao in higher dimensions \cite{Huang2020}.
Along this line, Huang and Li \cite{HuangL2020}, Ebenfelt, Xiao and Xu \cite{Ebenfelt2022}, Ganguly and Sinha \cite{Ganguly2022} studied other variations of Cheng's conjecture on Stein manifolds, and more generally on possibly singular Stein spaces with strongly pseudoconvex boundary.
For Reinhardt domains of finite type in $\mathbb{C}^{2}$, Fu and Wong \cite{Fu1997} proved the Bergman-Einstein condition  implies
that the domain must be strictly pseudoconvex. Hence, the domain must be a unit ball. Actually, a recent result given by Yuan  \cite{Yuan}
showed that Yau's conjecture is ture for all Reinhardt domains in $\mathbb{C}^{n}$ with smooth boundary.
See other variations of Cheng's
conjecture in Li \cite{Li2005,Li2009} and references therein.

For Thullen domain $D_{p}$,  it had been proved  by Cho \cite{Cho2021}   that
 the Bergman metric of $D_{p}\subset \mathbb{C}^{2}$
 is K\"ahler-Einstein if and only if $p=1$, i.e.,
 it is the unit ball in $\mathbb{C}^{2}$.
 Ebenfelt, Xiao and Xu's result in Proposition 1.10 of \cite{Ebenfelt2024} shows
it is also true for a more general case:
$
 D_{p,k}=\{(z,\zeta)\in \mathbb{B}^{n}\times\mathbb{C}^{k}:\|z\|^{2}+\|\zeta\|^{2p}<1, p>0\}.$
 In recently,
 Loi, Mossa and Zuddas \cite{Loi2025} proved Cheng's conjecture holds for Cartan-Hartogs domains
 which are a class of Hartogs domains over Cartan domains.
 The above result had been extended by  Mossa \cite{Mossa}  to a more general Hartogs domains  over bounded homogeneous domains.

 In general, let $\Omega$ be a domain in $\mathbb{C}^{n}$ and $\phi(z)$ be a continuous positive function on $\Omega$.
We define a Hartogs domain in $\mathbb{C}^{n+k}$ as
$\Omega_{\phi,k}=\{(z,\zeta)\in \mathbb{C}^{n+k}: z\in \Omega, \|\zeta\|^{2}< \phi(z)\}$.
It was proved by Cheng-Yau \cite{Cheng1980} and  Mok-Yau \cite{Mok1983} that a bounded domain $\Omega\subset \mathbb{C}^{n}$ is pseudoconvex if and only if it carries  a complete K\"ahler-Einstein metric with negative Ricci curvature.
Thus, for Cheng's
conjecture on  Hartogs domains, it suffices to restrict attention  to the  bounded pseudoconvex case.  It implies that $\Omega$ is  a bounded pseudoconvex domain and $-\log\phi$ is
 a strictly plurisubharmonic  function on $\Omega$. Take $h(z)=\phi^{-1}(z)$. We rewrite $\Omega_{\phi,k}$ as follows:
\begin{equation}\label{equ:B omega k}
B(\Omega, k)=\left\{(z,\zeta)\in \Omega\times \mathbb{C}^{k}:\|\zeta\|^{2}<h^{-1}(z)\right\}.
\end{equation}
It can be seen as a unit ball bundle in the  trivial vector bundle over  $\Omega$.
The function $h$  is the Hermitian  fiber metric and it induces the K\"ahler metric $g_{\Omega}$ by
$g_{i\bar j}=\frac{\partial^{2}}{\partial z_{i}\partial \overline{z}_{j}}\log h$.
It is worth to point out that $B(\Omega, k)$ is homogeneous  if and only if  it is biholomorphic to the unit ball $\mathbb{B}^{n+k}$ in $\mathbb{C}^{n+k}$ (see Remark \ref{rem:homo B}).
Hence,  Cheng and Yau's conjecture on Hartogs domains should be stated as follows:
 if the Bergman metric of a bounded pesudoconvex Hartogs domain is K\"ahler-Einstein, then the domain is biholomorphic to the ball?

Our second main work of this paper is to prove the   Cheng's conjecture  on $B(\Omega, k)$  $(k\geq2)$ over
  the simply-connected bounded pseudoconvex domain $\Omega$.
Our main idea is as follows:
Firstly, we show that the strictly pseudoconvex boundary  is spherical if the Bergman metric of $B(\Omega, k)$ is K\"ahler-Einstein.
This is an application of Huang and Li's result in \cite{HuangL2020}. Secondly, we prove
the Bochner-K\"ahler tensor of the K\"ahler metric induced by the boundary defining function of $B(\Omega, k)$ vanishes if and only if the base manifold is complex hyperbolic space. Here we need to combine  the work of Webster \cite{Webster1977}  on the  circle bundle with the relations between the curvature tensor, scalar curvature and Ricci tensor of the bundle metric and their counterparts for the base metric.
Finally, we prove the biholomorphic isometric  map  from the base space to  complex hyperbolic space induces  a biholomorphic mapping   from the $B(\Omega, k)$ to  $\mathbb{B}^{n+k}$.

The organization of this paper is as follows.
In Section \ref{2}, we introduce the definitions of several curvature tensors and the formulas describing their mutual relations.
In Section \ref{3}, we discuss the complete K\"ahler-Einstein metric on  the disk bundle and the uniform boundedness properties of several invariants.
In Section \ref{4}, we compute the holomorphic sectional curvature directly and calculate sectional curvatures using curvature identities. Using the uniform boundedness of invariants, we also study the negative upper and lower bounds for  these two curvatures.
In Section \ref{5}, we give a classification of a class of spherical-type sphere bundles.
As an application, we prove that Cheng's conjecture holds  for a class of bounded pseudoconvex Hartogs domains.

\section{Preliminary}\label{2}
Let $M$ be an $n$-dimensional complex manifold with local coordinates $z_{1},\cdots, z_{n}$.
Let
$g_{M}=\sum^{n}_{i,j=1} g_{i\bar j}dz_{i} \otimes d\bar{z}_{j}$
be a K\"ahler metric  on the holomorphic tangent bundle $T^{(1,0)}M$.
The components $R_{i\bar j k \bar l}$ of the curvature tensor $R_{M}$
associated with the metric connection are locally given by the formula
\begin{equation}\label{equ:CTR}
R_{i\bar j k \bar l}= -\frac{\partial^{2}g_{i\bar j}}{\partial z_{k}\partial \bar{z}_{l}}
+\sum_{p,q=1}^{n}g^{p\bar q}\frac{\partial g_{i\bar p}}{\partial z_{k}}\frac{\partial g_{q\bar j}}{\partial \bar z_{l}}.
\end{equation}
 Let $p$ be a point in $M$ and $Z=\sum^{n}_{i=1}Z_{i}\frac{\partial}{\partial z_{i}}\in T^{(1,0)}_{p}M$ be a non-zero holomorphic tangent vector at $p$.
 Then, holomorphic sectional curvature of $ds^{2}_{M}$ at $p$ in the direction $Z$ is
\begin{equation}
  \displaystyle{ H(g_{M},p,Z)=\frac{2R_{M}(Z,Z,Z,Z)}{(g_{M}(Z,Z))^{2}},}
 \end{equation}
 where $R_{M}(Z,Z,Z,Z)=\sum^{n}_{i,j,k,l=1}R_{i\bar j k \bar l}Z_{i} \overline{Z_{j}}Z_{k} \overline{ Z_{l}}$.
Denote by $G_{M}$ the matrix of  $g_{M}$, then
$$R_{M}=dz\left(-\bar\partial\partial G_{M}+\partial G_{M} (G_{M})^{-1}\overline{\partial G_{M}}^{t}\right)\overline{dz}^{t},$$
where
$\partial=\sum_{k=1}^{n}dz_{k}\frac{\partial}{\partial z_{k}}$, and $\bar{\partial}=\sum_{l=1}^{n}d\bar{z}_{l}\frac{\partial}{\partial \bar{z}_{l}}$ are differential operator on  $G_{M}$.

Let $g^{r}_{M}=\mathrm{Re} g_{M}$ be  the Riemannian metric on $M$.
Denote by $\nabla$ the  unique Levi-Civita connection.
Let $ u,  v, s, t\in TM$ be  non-zero tangent vectors fields.
The curvature operator is defined by $\mathcal{R}( u, v)s=\nabla_{ u}\nabla_{ v}s-\nabla_{ v}\nabla_{ u}s-\nabla_{[ v, u]}s$ and
the Riemannian curvature tensor is defined by
$R^{r}_{M}( u, v, s, t)=g^{r}_{M}(\mathcal{R}( s, t) u, v)$.
The sectional curvature of $ u,  v\in TM$ at $p\in M$ is
$$k(g^{r}_{M}, p,  u, v)=-\frac{R^{r}_{M}( u, v, u, v)}{\| u\wedge  v\|^{2}_{g^{r}_{M}}},$$
where
$\| u\wedge  v\|^{2}_{g^{r}_{M}}=\| u\|^{2}_{g^{r}_{M}}\| v\|^{2}_{g^{r}_{M}}-\langle  u, v \rangle_{g^{r}_{M}}^{2}$.
Define two complex  tangent vectors
$$Z=\frac{1}{2}( u-iJ u), W=\frac{1}{2}( v-iJ v) \in T_{(1,0)}M,$$
i.e., $ u=Z+\overline{Z}, J u=i(Z-\overline{Z}),  v=W+\overline{W}, J v=i(W-\overline{W})$.
It is well known that
\begin{equation}\label{1}
g^{r}_{M}( u, u)=g_{M}(Z,Z), \text{and}
-R^{r}_{M}( u,J u, u,J u)=2R_{M}(Z).
\end{equation}
This implies that the holomorphic sectional curvature of $g_{M}$ is
$$H(g_{M},p,Z)=k(g^{r}_{M}, p,  u, J u), \ \ \text{for} \ \ Z\in T_{(1,0)}M.$$

Now, we introduce \textbf{a tensor operator $\tau$} to illustrate the relationship between the holomorphic sectional curvature and the Riemannian sectional curvature.
 Define a operator
$\tau:T_{(1,0)}^{*}M\times T_{(0,1)}^{*}M\times T_{(1,0)}^{*}M\times T_{(0,1)}^{*}M  \longrightarrow T^{*}M\times T^{*}M$ as follows:
\begin{eqnarray}
 \tau(A)( u, v)&=&\frac{1}{32}[3A(Z+iW) +3A(Z-iW)-A(Z+W)
\nonumber\\
  &&  -A(Z-W) -4A(Z) -4A(W)],
 \end{eqnarray}
where $A$ is a $(2,2)$-tensor field and $A(Z)=A(Z,\overline{Z},Z, \overline{Z})$.
The following result can be obtained by the properties of the curvature (see also (4.3)-(4.7) in \cite{Jun Nie}  for more detains).
We prove it by a formula in \cite{Bishop1963} for  brevity.
\begin{lemma}\label{lem:tauR}
Let $(M, g_{M})$ be a K\"ahler manifold.
Let $R_{M}$ and $R^{r}_{M}$ be the curvature tensor in \eqref{equ:CTR} and   the Riemannian curvature tensor of $g_{M}$ respectively. Then
$\tau(R_{M})=-\frac{1}{2}R^{r}_{M}( u, v, u, v)$.
\end{lemma}
\begin{proof}
Let $B( u, v):=-R^{r}_{M}( u, v, u, v)$, it is shown by Bishop\cite{Bishop1963} that
\begin{eqnarray}
   B( u, v)&=&\frac{1}{32}[3Q( u+J v) +3Q( u-J v)-Q( u+ v)
\nonumber\\
  &&  -Q( u- v) -4Q( u) -4Q( v)],
 \end{eqnarray}
where $Q( u):=-R^{r}_{M}( u,J u, u,J u)$.
Notice that
\begin{eqnarray*}
&&\frac{1}{2}( u+J v-iJ( u+J v))= Z+iW, \  \frac{1}{2}( u-J v-iJ( u-J v))= Z-iW,\\
&&\frac{1}{2}( u+ v-iJ( u+ v))= Z+W, \ \ \ \ \ \ \frac{1}{2}( u- v-iJ( u- v))= Z-W.
\end{eqnarray*}

By \eqref{1}, we get $Q( u+J v)=2R_{M}(Z+iW)$, $Q( u-J v)=2R_{M}(Z-iW)$, $Q( u+ v)=2R_{M}(Z+W)$,  $Q( u- v)=2R_{M}(Z-W)$,   $Q( u)=2R_{M}(Z)$,  $Q( v)=2R_{M}(W)$.
Hence,
 we  have the formula
\begin{eqnarray}\label{Relation2}
   -R^{r}_{M}( u, v, u, v)&=&\frac{1}{16}[3R_{M}(Z+iW) +3R_{M}(Z-iW)-R_{M}(Z+W)
\nonumber\\
  &&  -R_{M}(Z-W) -4R_{M}(Z) -4R_{M}(W)].
 \end{eqnarray}
The conclusion  follows from it.
 \end{proof}

\begin{lemma}\label{lem:taug}
 Let $(M, g_{M})$ be a K\"ahler manifold
and $g^{r}_{M}=\mathrm{Re} g_{M}$ be the Riemannian metric. Then
\begin{eqnarray}
\tau(g_{M}\otimes g_{M})( u, v)
&=&\frac{1}{4}\left(\| u\wedge  v\|^{2}_{g^{r}_{M}}+3\langle  u,J v \rangle_{g^{r}_{M}}^{2}\right),
\end{eqnarray}
 where $\| u\wedge  v\|^{2}_{g^{r}_{M}}=\| u\|^{2}_{g^{r}_{M}}\| v\|^{2}_{g^{r}_{M}}-\langle  u, v \rangle_{g^{r}_{M}}^{2}$.
Moreover, we have
$$\frac{1}{4}\| u\wedge  v\|^{2}_{g^{r}_{M}}\leq  \tau(g_{M}\otimes g_{M})( u, v)\leq \| u\wedge  v\|^{2}_{g^{r}_{M}}.$$
\end{lemma}
\begin{proof}
It can be obtained by a complicated but straightforward calculation.
\end{proof}

\section{The complete K\"ahler-Einstein metric on the disk bundle}\label{3}
Let  $\pi:(L, h)\rightarrow M$ be a negative Hermitian line bundle over an $n$-dimensional complex manifold $M$.
Assume  that the curvature of the dual bundle $(L^*, h^{-1})$ induces a K\"ahler metric $g_{M}$ on $M$, whose the K\"ahler form is given by $\omega_{M}=\frac{\sqrt{-1}}{2}\partial \bar\partial\log h$.
 The disk bundle is defined by
 \begin{equation}\label{D1}
 D(L) := \{\zeta  \in L : |\zeta|_{h}^{2} < 1\},
 \end{equation}
where $|\zeta|_{h}$ is the norm of $\zeta$ with respect to the metric $h$.
Recall that for the K\"ahler metric $g_{M}=g_{i\bar j}dz_{i}\otimes d\bar{z}_{j}$, the components of the Ricci tensor $\mathrm{Ric}_{M}$
are given by  $$R_{i\bar j}=-\frac{\partial^{2}}{\partial z_{i}\partial \overline{z}_{j}}\log\det (g_{k\bar l}).$$
A K\"ahler manifold is called Einstein if $R_{i\bar j}=\lambda g_{i\bar j}$.
Let $P(y)$ be the characteristic polynomial of the linear operator:
$$\frac{2}{n+2}\mathrm{Ric}_{M}\cdot g_{M}^{-1}:T_{p}^{(1,0)}M\rightarrow T_{p}^{(1,0)}M \ \text{at} \ p\in M.$$
If $(M, g_{M})$ has constant Ricci eigenvalues, then
$$P(y)=\det\left(yI_{n}-\frac{2}{n+2}\mathrm{Ric}\cdot g_{M}^{-1}\right)=\prod_{i=1}^{n}(y-\frac{2\lambda_{i}}{n+2}),$$ where  $\lambda_{1}\leq\lambda_{2}\leq\cdots<\lambda_{n}<1$ are the Ricci eigenvalues of $(M, g_{M})$.
The following result given by  Ebenfelt, Xiao and Xu provides an explicit formula of the complete K\"ahler-Einstein metric.
\begin{lemma}\label{lemma:1}\cite{Ebenfelt2023}
Assume that the $n$-dimensional complete K\"ahler manifold $(M, g_{M})$ has constant Ricci eigenvalues
and every Ricci eigenvalue is strictly less than one, then the disk
bundle $D(L)$ admits a unique complete K\"ahler-Einstein metric $\widetilde{g}_{D}$ with Ricci curvature
$-(m + 1)$, where $m=n+1$.
Moreover, this metric is induced by the following K\"ahler form:
$$\widetilde{\omega}_{D}=-\frac{1}{m+1}\pi^* Ric_{M}+\frac{1}{m+1}\pi^*\omega_{M}-\frac{\sqrt{-1}}{2}\partial \overline{\partial}\log \phi(|\zeta|_{h}),$$
where $\pi : L \rightarrow M$ is the canonical fiber
projection of the line bundle, $\omega_{M}$ and $\mathrm{Ric}_{M}$ are  the K\"ahler form and the Ricci form of $(M, g)$ respectively,  $\phi : (-1, 1) \rightarrow R^{+}$ is an even real analytic function that depends only on the
characteristic polynomial of the Ricci endomorphism.
In the local coordinates,  the function
$$\widetilde{\psi}(z,\zeta)=\frac{1}{m+1} \log \det (g_{i\overline{j}}(z))+\frac{1}{m+1}\log h(z)-\log\phi(|\zeta|h^{\frac{1}{2}}(z))$$
is  a  K\"ahler potential function.
\end{lemma}

Let $m=n+1$ and consider the equation
$$rZ'\widehat{P}(Z)+\widehat{Q}(Z)=0, Z(1)=0,$$
where $\widehat{P}(x)=x^{m-1}P(x^{-1}), \ \widehat{Q}(x)=x^{m+1}Q(x^{-1})$
 for $P(y)=\prod_{i=1}^{n}(y-\frac{2\lambda_{i}}{m+1})$ and an $m+1$ degree monic polynomial $Q(y)$ satisfying $\frac{dQ}{dy}=(m+1)yP(y)$, $Q(\frac{2}{m+1})=0.$
 By Proposition 2.7 of \cite{Ebenfelt2023}, there exists a unique real analytic function $Z$ on $[-1, 1]$.
 Moreover, $Z$ is an even function satisfying $Z(0)=\frac{m+1}{2}$ and $Z'(0)=0, Z'(r)<0$ on $(0,1]$ and $Z'(1)=-1, Z''(0)<0$.
 $Z(r)\in(0,\frac{m+1}{2})$ on $(-1,0)\cup(0,1)$.
The function   $\phi$  is defined by
\begin{equation}\label{equ:phi}
\phi(r)=2\left(\frac{r}{-Z'\widehat{P}(Z)}\right)^{\frac{1}{m+1}}Z.
\end{equation}
It had been proved that
\begin{enumerate}
  \item[(a)] $\phi$ is real analytic on $[-1,1]$;
  \item[(b)] $\phi$ is an  even function satisfying $\phi>0$ on $(-1,1)$, $\phi(1)=0$;
  \item[(c)] $\phi$ satisfies that
$
(m+1)rZ\phi'+(m+1-2Z)\phi=0
$
with $\phi'(1)=-2$.
\end{enumerate}

Define $y(r)=\frac{1}{Z(r)}$ for $r\in [-1 , 1]$. Item (c) implies that $y=\frac{2}{m+1}-r\frac{\phi'}{\phi}$.
Notice that $y^{m-1}\widehat{P}(y^{-1})=P(y)$ and $Z'=\frac{y'}{y^{2}}$.
Equation  \eqref{equ:phi} turns to be
\begin{equation}\label{equ:122}
y'P(y)=2^{m+1}r\phi^{-(m+1)}(r), \ \text{on} \ [-1,1].
\end{equation}

Take the logarithm of the above equation and then take its derivative, we have
\begin{eqnarray*}
\left(\log [y'P(y)]\right)'&=&\frac{1}{r}-(m+1)\frac{\phi'}{\phi}\\
&=&\frac{1}{r}+(m+1)\frac{1}{r}(y-\frac{2}{m+1})\\
&=&\frac{1}{r}((m+1)y-1).
\end{eqnarray*}
It is equivalent to
\begin{equation}\label{equ:a}
\left(r[y'P(y)]\right)'=(m+1)yy'P(y).
\end{equation}
\begin{corollary}
Let $(L, h)$ be a negative line bundle over $M$ such that the dual bundle $(L^{*}, h^{-1})$ induces a complete K\"ahler metric $g_{M}$.
If $(M, g_{M})$ is a complete K\"ahler-Einstein manifold whose Ricci curvature satisfying $\lambda_{0}<1$,
then the disk bundle $D(L)$ admits a unique complete K\"ahler-Einstein metric  $\widetilde{g}_{D}$ with Ricci curvature $-(m+1)$.
The following function
\begin{equation}\label{g fct}
\widetilde{\psi}=\frac{1}{m+1}\log Y^{m-1}Y'+
\frac{1-\lambda_{0}}{m+1}\log h
\end{equation}
is a K\"ahler potential function of  its unique complete K\"ahler-Einstein metric,
 where $Y$ satisfies the original  differential equation:
\begin{equation}\label{ode}
\left\{
  \begin{array}{ll}
    XY^{m-1}Y'=Y^{m+1}+\frac{\lambda_{0}}{m}Y^{m}+C  \hbox{;} \\
   Y(0)=\frac{1-\lambda_{0}}{m+1}  \hbox{,}
  \end{array}
\right.
\end{equation}
and  $C=\frac{-(\lambda_{0}+m)(1-\lambda_{0})^{m}}{m(m+1)^{m+1}}$. Here,  $X=|\zeta|^{2}h(z)$, and
for any $(z,\zeta)\in D$, $X\in[0,1]$, $Y(X)\geq \frac{1-\lambda_{0}}{m+1},$  $ Y'(X)>0$ on $(0,1)$, and $Y'(0)=0$.
Moreover,  for $1\leq i, j\leq m$, the metric $\widetilde{g}_{D}$ can be expressed by
\begin{eqnarray*}\label{gij}
\widetilde{g}_{i\bar j}&=&\frac{Y'}{X}X_{i}X_{\bar j}+Y(\log X)_{i\bar j}.
\end{eqnarray*}

\end{corollary}
\begin{proof}

(i)
For local
coordinates $(U, z)$ of $M$ and a natural local free frame $\{e_{U}\}$ on $U$,  there exists a system of linear holomorphic coordinates $\zeta$ on each fiber of $\pi^{-1}(U)$,
 so  $(z, \zeta)$ is a holomorphic coordinate system of $\pi^{-1}(U)$.
One can locally represent the Hermitian structure on $L$ by a positive function $h(z)$  on $U$ such that its Hermitian form on $L$  can be expressed as
$|\zeta|_{h}^{2}=h(z)\zeta\bar{\zeta}$. Locally, $D(L)$ can be written as follows:
$$D(L)\cap \pi^{-1}(U)=\left\{(z,\zeta)\in U\times \mathbb{C}:|\zeta|^{2}h(z)<1\right\}.$$
Suppose that  $(M, g_{M})$ is K\"ahler-Einstein, i.e., $\lambda_{i}=\lambda_{0}$ for $i=1,2, \cdots, n$. Thus we have
 $$P(y)=(y-\frac{2\lambda_{0}}{m+1})^{m-1}.$$
Define $X=|\zeta|^{2}_{h}\in [0, 1)$ and   $Y(X)=\frac{1}{2}\left(y(X^{\frac{1}{2}})-\frac{2\lambda_{0}}{m+1}\right)$.
By Lemma \ref{lemma:1}, in the local coordinates,
\begin{equation}\label{equ:psi1}
\widetilde{\psi}
=\frac{1-\lambda_{0}}{m+1}\log h-\log \phi(|\zeta|_{h})
\end{equation}
is one of the K\"ahler potential functions of $\widetilde{g}_{D}$.
Take the logarithm of the  equation \eqref{equ:122},  we have
 $$-\log \phi(r)=\frac{1}{m+1}\left(\log [y'P(y)]-\log r\right)-\log 2 .$$
Take $r=X^{\frac{1}{2}}$, by the definition \eqref{equ:phi} of $\phi$, we get
\begin{eqnarray*}
-\log \phi(X^{\frac{1}{2}})&=&\frac{1}{m+1}\left(\log [2^{m+1}X^{\frac{1}{2}}Y'Y^{n}]-\log X^{\frac{1}{2}}\right)-\log 2\\
&=&\frac{1}{m+1}\log [Y'Y^{n}].
\end{eqnarray*}
Insert it into \eqref{equ:psi1}, we obtain that
$$\widetilde{\psi}
=\frac{1}{m+1}\log [Y'Y^{n}]+\frac{1-\lambda_{0}}{m+1}\log h(z).$$

(ii)
Take $r=X^{\frac{1}{2}}$ in \eqref{equ:a}. Since
$y'(X^{\frac{1}{2}})=4X^{\frac{1}{2}}Y'(X)$ and $P(y(X^{\frac{1}{2}}))=2^{m-1}Y^{m-1}$, we can get
\begin{eqnarray*}
&&2^{m+1}\left([XY'(X)Y^{m-1}]\right)'2X^{\frac{1}{2}}\\
&=&(m+1)(2Y+\frac{2\lambda_{0}}{m+1})[2^{m+1}X^{\frac{1}{2}}Y'(X)Y^{m-1}]\\
&=&2(m+1)[2^{m+1}X^{\frac{1}{2}}Y'(X)Y^{m}]+2\lambda_{0}[2^{m+1}X^{\frac{1}{2}}Y'(X)Y^{m-1}].
\end{eqnarray*}
This implies that
$[XY'(X)Y^{m-1}]'=[Y^{m+1}]'+\frac{\lambda_{0}}{m}[Y^{m}]'.$
It is equivalent  to
$$XY'(X)Y^{m-1}=Y^{m+1}+\frac{\lambda_{0}}{m}Y^{m}+C. $$
Let $X=0$, we know $Y(0)=\frac{1}{2}y(0)-\frac{\lambda_{0}}{m+1}=\frac{1-\lambda_{0}}{m+1}>0 .$
Thus, the constant
$$C=-Y^{m}(0)(Y(0)+\frac{\lambda_{0}}{m})=-\frac{(\lambda_{0}+m)(1-\lambda_{0})^{m}}{m(m+1)^{m+1}}.$$
Since $y'(r)=-\frac{Z'(r)}{Z^{2}(r)}$,   we know
$Y'(X)=\frac{1}{4}y'(X^{\frac{1}{2}})X^{-\frac{1}{2}}=-\frac{1}{4}\frac{X^{-\frac{1}{2}}Z'(X^{\frac{1}{2}})}{Z^{2}(X^{\frac{1}{2}})}$.
Hence,  $Y'(0)=0$ and  $Y'(X)>0$ on $(0, 1)$.

(iii) By a simple transformation,  \eqref{ode} turns to be
$$
\frac{[Y^{m-1}Y']'}{Y^{m-1}Y'}=(m+1)\frac{Y}{X}-\frac{1-\lambda_{0}}{X}.
$$
Insert it into the following computation of the metric, for $1\leq i, j\leq m$, we get
\begin{eqnarray*}
\widetilde{g}_{i\bar j}&=&\frac{1}{m+1}\left(\frac{[Y^{m-1}Y']'}{[Y^{m-1}Y']}X_{i}+(1-\lambda_{0})(\log X)_{i}\right)_{\bar j}\\
&=&\frac{1}{m+1}\left(((m+1)\frac{Y}{X}-\frac{1-\lambda_{0}}{X})X_{i}+(1-\lambda_{0})(\log X)_{i}\right)_{\bar j}\\
&=&(\frac{YX_{i}}{X})_{\bar j}
=\frac{Y'}{X}X_{i}X_{\bar j}+Y(\log X)_{i\bar j}.
\end{eqnarray*}
\end{proof}

Notice that $\lambda_{0}\leq -m$ iff $C\geq0$, and  $\frac{C}{Y^{m+1}(0)}=-\frac{\lambda_{0}+m}{m(1-\lambda_{0})}$.
Hence,   for $X\in [0,1)$, we have
\begin{equation}\label{inequ1}
\frac{C}{Y^{m+1}(X)}\in
\left\{
\begin{array}{ll}
              \left(\left.0,-\frac{\lambda_{0}+m}{m(1-\lambda_{0})}\right]\right., & \hbox{ if $\lambda_{0}\leq -m$;} \\
\\
              \left[\left.-\frac{\lambda_{0}+m}{m(1-\lambda_{0})}, 0\right)\right., & \hbox{ if $\lambda_{0}>- m$.}
            \end{array}
          \right.
\end{equation}
This implies the following results:
\begin{equation}\label{inequ0}
1+\frac{C}{Y^{m+1}(X)}\in
\left\{
\begin{array}{ll}
              \left(\left.1, \frac{-(m+1)\lambda_{0}}{m(1-\lambda_{0})}\right]\right., & \hbox{ if $\lambda_{0}\leq -m$;} \\
\\
              \left[\left.\frac{-(m+1)\lambda_{0}}{m(1-\lambda_{0})}, 1\right)\right., & \hbox{ if $\lambda_{0}>- m$.}
            \end{array}
          \right.
\end{equation}
\begin{equation}\label{inequ2}
1-\frac{mC}{Y^{m+1}(X)}\in
\left\{
\begin{array}{ll}
              \left[\left.\frac{m+1}{1-\lambda_{0}}, 1\right)\right., & \hbox{ if $\lambda_{0}\leq -m$;} \\
\\
              \left(\left.1, \frac{m+1}{1-\lambda_{0}}\right]\right., & \hbox{ if $\lambda_{0}>- m$.}
            \end{array}
          \right.
\end{equation}
\begin{equation}\label{inequ3}
2+\frac{m(m-1)C}{Y^{m+1}(X)}\in
\left\{
\begin{array}{ll}
              \left(\left. 2, -\frac{(m+1)(\lambda_{0}+m-2)}{1-\lambda_{0}}\right]\right., & \hbox{ if $\lambda_{0}\leq -m$;} \\
\\
             \left[\left.-\frac{(m+1)(\lambda_{0}+m-2)}{1-\lambda_{0}}, 2\right)\right., & \hbox{ if $\lambda_{0}>- m$.}
            \end{array}
          \right.
\end{equation}
 \begin{lemma}\label{lem:totals}
 Let $(L, h)$ be a negative line bundle over $M$ such that the dual bundle $(L^{*}, h^{-1})$ induces a complete K\"ahler-Einstein  metric $g_{M}$
 whose Ricci curvature satisfying $\lambda_{0}<1$.
Let $\widetilde{g}_{D}$  be the unique complete K\"ahler-Einstein metric   with Ricci curvature $-(m+1)$ on the disk bundle $D(L)$.
 Then
$(M, \alpha g_{ M})$ is a totally geodesic submanifold of $(D(L), \widetilde{g}_{D})$,
where  $\alpha=\frac{1-\lambda_{0}}{m+1}$.
\end{lemma}
\begin{proof}
By \eqref{g fct}, we have $\widetilde{\psi}|_{M}=\alpha\log h$. This implies $(M, \alpha g_{M})$ is a K\"ahler submanifold of $(D(L), \widetilde{g}_{D})$.
Let $f: M\rightarrow D(L)$ be the inclusion map.        Since  $M$ is the fixed point set of the circle group $S^{1}$
acting by $(e^{\sqrt{-1}\theta};(z, \zeta))\rightarrow (z, e^{\sqrt{-1}\theta} \zeta)$.
Hence,
$f:(M, \alpha g_{M})\rightarrow(D(L), \widetilde{g}_{D})$ is a totally  geodesic holomorphic isometric
embedding.
\end{proof}

\section{The estimate of the curvatures}\label{4}
\subsection{A local coordinate}

Let $\pi$ be the project mapping from $L$ to $M$. The project point of $p_{0}$ is denoted by $p'_{0}$.
 We take the geodesic
coordinate $(U, z)$ around  $p'_{0}$.  The metric $g_{M}$ is denoted by  $\sum_{j,k=1}^{n} g_{j\bar k}dz_{j} \otimes d\overline{z}_{k}$.
 At the point $p'_{0}$, we have that $g_{j\bar k}=\delta_{j\bar k}$,  and all first derivatives of $g_{j\bar k}$ are zero.
Let $z_{0}$ be the coordinate of  $p'_{0}$,
and $\varphi$ be a K\"ahler potential of $g_{M}$ in $U$.
Let $\varphi(z,w)$ be the polarized function of $\varphi$ on $U\times \mathrm{conj}(U)$.
Define another K\"ahler potential function by
 $$\psi(z)=\varphi(z,\bar z)-\varphi(z,\bar z_{0})-\varphi(z_{0}, \bar z)+\varphi(z_{0},\bar z_{0}),$$
 such that $\frac{\sqrt{-1}}{2}\partial\bar\partial\psi(z)=\frac{\sqrt{-1}}{2}\partial\bar\partial\varphi(z)=\omega_{M}$.
Recall that $\partial\bar\partial\log h=\partial\bar\partial\psi$.
It is equivalent to
$h|e^{f}|^{2}=e^{\psi}$
for a certain holomorphic function $f$ in $U$.
Choose a local free frame $e_{U}$ such that
$h=e^{\psi}.$
Hence we have $$h(z_{0})=1, h_{j}(z_{0})=e^{\psi}\psi_{j}(z_{0})=0, h_{\bar k}(z_{0})=e^{\psi}\psi_{\bar k}(z_{0})=0,
h_{j\bar k}(z_{0})=\delta_{j\bar k},$$
where $h_j=\frac{\partial h}{\partial z_{j}}$, $h_{\bar{k}}=\frac{\partial h}{\partial \bar{z}_{k}}$, $h_{j\bar k}=\frac{\partial^2 h}{\partial z_{j}\partial \bar{z}_k}$. For the local geodesic
coordinates $(U, z)$ of $M$ and a natural local free frame $\{e_{U}\}$ on $U$, $D(L)$ can be written as follows:
$$D(L)\cap \pi^{-1}(U)=\left\{(z,\zeta)\in U\times \mathbb{C}:|\zeta|^{2}h(z)<1\right\}.$$

Recall that $X=|\zeta|^{2}_{h}=|\zeta|^{2}h(z)$ on $D(L)\cap \pi^{-1}(U)$.
In the following, we will give the partial derivatives  of $X$ at the fixed point $(z_{0}, \zeta)$ which will beneficial to compute
the curvatures.

Suppose that $1\leq i, j, k, l\leq n=m-1$.
The first-order partial derivatives  of $X$ at the fixed point $(z_{0}, \zeta)$ are
$ X_{i}|_{z=z_{0}}=X_{\bar j}|_{z=z_{0}}=0$  and
$X_{m}|_{z=z_{0}}=\overline{ \zeta }, X_{\overline{m}}|_{z=z_{0}}= \zeta$.
Thus, we have
\begin{equation}
X|_{z=z_{0}}=|\zeta |^{2}, \partial X|_{z=z_{0}}=\overline{\zeta}d\zeta, \ \overline{\partial} X|_{z=z_{0}}=\zeta d\overline{\zeta };
\end{equation}

Take the partial derivative again.  At the fixed point $(z_{0}, \zeta)$ we have the second-order partial derivatives.
For convenience, we omit the subscripts of the coordinate points.  They are
$X_{ij}=0, X_{i\bar j}=|\zeta|^{2}h_{i\bar j}, X_{\bar j i}=|\zeta|^{2}h_{\bar j i}=\overline{X}_{\overline{i}j},$
$X_{i \overline{m}}=X_{mj}=X_{mm}=X_{\overline{m}i}= 0, X_{m\overline{m}}=X_{\overline{m}m}= 1.$
It implies that at the  point $(z_{0}, \zeta)$, there are
$$\overline{\partial} X_{i}=\sum_{l=1}^{n} X_{i\overline{l}}d\overline{z}_{l},
\overline{\partial} X_{\bar j}=0, \partial X_{\overline{j}}=\sum_{k=1}^{n} X_{\overline{j}k}dz_{k},$$ and
$\overline{\partial} X_{m}=d\overline{\zeta },
\partial X_{m}=0,
\overline{\partial} X_{\overline{m}}=0,
\partial X_{\overline{m}}=d\zeta , \overline{\partial} (X_{m}X_{\overline{j}})=0,
\partial (X_{m}X_{\overline{j}})=\overline{\zeta}X_{\bar{j}k}dz_{k}.$

Moreover, at any point $(z, \zeta)$, we have
$
(\log X)_{m \overline{m}}
\equiv(\log X)_{m\bar j }\equiv(\log X)_{i\overline{m} }\equiv0,$
and at the fixed point $(z_{0}, \zeta)$, we have $(\log X)_{i\bar j }=h_{i\bar j }$,
 $$\overline{\partial}\partial X_{\overline{j}}=\sum_{k=1}^{n}\zeta h_{\overline{j}k}d\overline{\zeta }dz_{k},
\overline{\partial} \partial X_{m}= \sum_{k,l=1}^{n} \overline{\zeta }h_{k\overline{l}}dz_{k}d\overline{z}_{l}=\overline{\zeta }dz T_{z_{0}} \overline{dz}^{t},$$
$$\overline{\partial} \partial (X_{m} X_{\overline{m}})
=2|\zeta |^{2}dz T_{z_{0}}\overline{dz}^{t}+|d\zeta |^{2},
\overline{\partial} \partial \log X=
dz T_{z_{0}} \overline{dz}^{t}, \overline{\partial} \partial X
=|\zeta|^{2}dz T_{z_{0}} \overline{dz}^{t}+d\zeta d\overline{\zeta },$$
where   $T$ is the matrix of $g_{M}$, and at the point $z_{0}$, $T_{z_{0}}$ is an identity matrix.
\subsection{Holomorphic sectional curvature}
Denote by $G$ the matrix of  the complete K\"ahler-Einstein metric. The curvature tensor is
\begin{equation}\label{equ:formula R}
R_{D}=d\eta\left(-\bar\partial\partial G+\partial G (G)^{-1}\overline{\partial G}^{t}\right)\overline{d\eta}^{t},
\end{equation}
where $\partial=\sum_{k=1}^{m}d\eta_{k}\frac{\partial}{\partial \eta_{k}}$, and $\bar{\partial}=\sum_{l=1}^{m}d\bar{\eta}_{l}\frac{\partial}{\partial \bar{\eta}_{l}}$ are differential operator and
\begin{eqnarray}\label{G}
G &=&\left(\frac{Y'}{X}X_{i}X_{\bar j}+Y(\log X)_{i\bar j}\right),
\end{eqnarray}
where $1\leq i, j \leq m$.
Now, we compute the value  $R_{D}$ at $(z_{0}, \zeta)$.
\begin{lemma}\label{cur 0}
The curvature tensor $R_{D}$ at $(z_{0}, \zeta )$
is
\begin{eqnarray*}
R_{D}(z_{0}, \zeta )
&=&-2\left(1+\frac{C}{Y^{m+1}}+\frac{\lambda_{0}}{mY}\right)\left\{\sum_{i,j=1}^{n}Y(\log X)_{i\bar j}dz_{i}\overline{dz_{j}}\right\}^{2}+YR_{ M}|_{z=z_{0}}\\
&&-4\left(1-\frac{mC}{Y^{m+1}}\right)\left\{\sum_{i,j=1}^{n}Y(\log X)_{i\bar j}dz_{i}\overline{dz_{j}}\right\}\left\{Y'd\zeta d\overline{\zeta }\right\}\\
&&-\left(2+\frac{m(m-1)C}{Y^{m+1}}\right)\left\{Y'd\zeta d\overline{\zeta }\right\}^{2}.
 \end{eqnarray*}
\end{lemma}

\begin{proof}
By \eqref{G},  we know
\begin{equation}\label{equ:Hession of solution 1}
G_{D}=\left(
  \begin{array}{ccc}
 \frac{Y'}{X}X_{m}X_{\overline{m}} &  & \frac{Y'}{X}X_{m}X_{\bar j} \\
  & \\
   \frac{Y'}{X}X_{i}X_{\overline{m}} & &  \ \ \ \frac{Y'}{X}X_{i}X_{\bar j}+Y(\log X)_{i\bar j} \\
  \end{array}
\right),
\end{equation}
where $1\leq j, k\leq n$.
At the point $\eta_{0}=(z_{0}, \zeta )$, we have
\begin{equation}
G_{D}=\left(
  \begin{array}{cc}
 Y' & 0 \\

  0  &  \ \ \ Y(\log X)_{i\bar j} \\
  \end{array}
\right)
\text{and} \
G_{D}^{-1}=\left(
  \begin{array}{cc}
 \frac{1}{Y'} & 0 \\

  0  &  \ \ \ \frac{1}{Y}(\log X)^{-1}_{i\bar j} \\
  \end{array}
\right).
\end{equation}
For convenience, we define
\begin{equation}
\partial G_{D}=\left(
  \begin{array}{cc}
 \partial G_{11} &  \partial G_{12} \\

   \partial G_{21}  &   \partial G_{22} \\
  \end{array}
\right).
\end{equation}
By a direct computation, at $\eta_{0}=(z_{0}, \zeta )$, we have
$$
\partial G_{11}
=Y''\overline{\zeta }d\zeta ,
\partial G_{12}
=Y' \overline{\zeta } dz T_{z_{0}},
\partial G_{21}=0,
\partial G_{22}
= Y' \bar \zeta  d\zeta  T_{z_{0}}+Y \partial T_{z_{0}},
$$
where   $T$ is the matrix of $g_{M}$, and at the point $z_{0}$, $T_{z_{0}}$ is an identity matrix. Moreover, we have
\begin{eqnarray*}
\overline{\partial} \partial G_{11}
&=&(Y'''X+Y'')|d\zeta |^{2}+(Y''X+Y')dzT_{z_{0}}\overline{dz}^{t},\\
\overline{\partial}\partial G_{12}
&=&(XY''+Y')dz T_{z_{0}} d\overline{\zeta }+\overline{\zeta }Y'd z \overline{\partial} T_{z_{0}},
\overline{\partial}\partial G_{21}=\overline{\overline{\partial}\partial G_{12}}|_{z=z_{0}}^{t},\\
\overline{\partial}\partial G_{22}
&=& (XY''+Y')|d\zeta |^{2}T_{z_{0}}+XY' dz T_{z_{0}} \overline{dz}^{t} T_{z_{0}}+XY'\overline{T_{z_{0}}}^{t}d\overline{z}^{t}dzT_{z_{0}}\\
&&+
Y'\left(\overline{\zeta }d\zeta \overline{\partial}T_{z_{0}}+\zeta d\overline{\zeta }\partial T_{z_{0}}\right)+Y\overline{\partial}\partial T_{z_{0}}.
\end{eqnarray*}

Define
\begin{equation*}
-\bar\partial\partial G+\partial G (G)^{-1}\overline{\partial G}^{t}=\left(
  \begin{array}{cc}
  R_{11} &   R_{12} \\

    R_{21}  &  R_{22} \\
  \end{array}
\right),
\end{equation*}
then
$R_{D}=d\zeta  R_{11} d\overline{\zeta }+d\zeta  R_{12} \overline{dz}^{t}+dz R_{21} d\overline{\zeta }+dz R_{11} \overline{dz}^{t}$.
By the discussion above, at the point $\eta_{0}=(z_{0}, \zeta )$, we know
\begin{eqnarray*}
R_{11}&=&(\frac{XY'^{2}}{Y}-XY''-Y')dzT_{z_{0}}\overline{dz}^{t}+(\frac{XY''^{2}}{Y'}-XY'''-Y'')|d\zeta |^{2},\\
R_{12}&=&(\frac{XY'^{2}}{Y}-XY''-Y')dzT_{z_{0}}d\overline{\zeta }, R_{21}=\overline{R_{12}}^{t},\\
R_{22}&=&(\frac{XY'^{2}}{Y}-XY''-Y')|d\zeta |^{2}T_{z_{0}}-XY'dzT_{z_{0}}\overline{dz}^{t}\overline{T_{z_{0}}}^{t}-XY'\overline{T_{z_{0}}}^{t}\overline{dz}^{t}dzT_{z_{0}}\\
&&+
Y(-\overline{\partial}\partial T_{z_{0}}+\partial T_{z_{0}} T_{z_{0}}^{-1}T_{z_{0}} \overline{\partial T_{z_{0}}}^{t}).
\end{eqnarray*}
By \eqref{ode}, we have $XY'=Y^{2}+\frac{\lambda_{0}}{m}Y+\frac{C}{Y^{m-1}}$, $\frac{XY'^{2}}{Y}-XY''-Y'=-YY'(1-\frac{mC}{Y^{m+1}})$ and
\begin{eqnarray*}
\frac{XY''^{2}}{Y'}-XY'''-Y''&=&-(2+\frac{m(m-1)C}{Y^{m+1}})Y'^{2}.
\end{eqnarray*}
Denote by $R_{M}$  the curvature tensor of $g_{M}$.
 The formula  \eqref{equ:formula R} shows that
\begin{eqnarray*}
R_{D}&=&-2Y^{2}\left(1+\frac{C}{Y^{m+1}}+\frac{\lambda_{0}}{mY}\right)(dz T_{z_{0}} \overline{dz}^{t})^{2}+Y R_{M}
\\
&&-4YY'(1-\frac{mC}{Y^{m+1}})dz T_{z_{0}} \overline{dz}^{t}|d\zeta |^{2}-(2+\frac{m(m-1)C}{Y^{m+1}})Y'^{2}|d\zeta |^{4}
 \end{eqnarray*}
at the point $\eta_{0}=(z_{0}, \zeta )$.
The  conclusion follows from it.
\end{proof}
\begin{lemma}
The curvature tensor at $\eta=(z,\zeta )$ can be expressed by
\begin{eqnarray*}
R_{D}
&=&-2\left(\frac{(m+1)C}{Y^{m+1}}+\frac{\lambda_{0}}{ mY}\right)\left\{\sum_{i,j=1}^{m}Y(\log X)_{i\bar j}d\eta_{i}d\overline{\eta}_{j}\right\}^{2}+YR_{ M}\\
&&-2\left(1-\frac{mC}{Y^{m+1}}\right)\left\{\sum_{i,j=1}^{m}g_{i\bar j}d\eta_{i}d\overline{\eta}_{j}\right\}^{2}-\left(\frac{m(m+1)C}{Y^{m+1}}\right)\left\{\sum_{i,j=1}^{m}\frac{Y'}{X}X_{i}X_{\bar j}d\eta_{i}d\overline{\eta}_{j}\right\}^{2}.
 \end{eqnarray*}
\end{lemma}
\begin{proof}
Let $(z, U(p))$ be the local normal coordinate system around $p=z_{0}$.
We choose a fixed point $q$ in $U(p)$, and denote its coordinate  by $\hat{z}_{0}$.
Take a new local normal coordinate $(w, \widetilde{U}(q))$ around $q$.  Denote by $w_{0}$ the new coordinate the point $q$.
Then we have
 $$g_{M}(q)=dwT_{w_{0}}\overline{dw}^{t}=dzT_{\hat{z}_{0}}\overline{dz}^{t}=\sum_{i,j=1}^{n}(\log h)_{i\bar j}(\hat{z}_{0})dz_{i}d\overline{z}_{j}.$$
Notice that for $1\leq i,j\leq n$, $(\log X)_{i \overline{j}}=(\log h)_{i\bar j}$,
$(\log X)_{m \overline{m}}
\equiv(\log X)_{m\bar j }\equiv(\log X)_{i\overline{m} }\equiv0,$ we get
$$g_{M}(q)=\sum_{i,j=1}^{m}(\log X)_{i\bar j}(\hat{z}_{0}, \zeta)d\eta_{i}d\overline{\eta}_{j}.$$
Let $\Phi$  be the  local coordinate transform from $(z, U(p))$ to $(w, \widetilde{U}(q))$ around $q$, i.e.,
$w=\Phi(z)$, $w_{0}=\Phi(\hat{z}_{0}).$
Then we have
$\det h_{i\bar{j}}(z)=\det \widetilde{h}_{i\bar{j}}(w)|\det\frac{\partial\Phi}{\partial z}|^{2},$
and
$$\det h_{i\bar{j}}(\hat{z}_{0})=\det\widetilde{h}_{i\bar{j}}(w_{0})\left|\det\frac{\partial\Phi(\hat{z}_{0})}{\partial z}\right|^{2}=\left|\det\frac{\partial\Phi(\hat{z}_{0})}{\partial z}\right|^{2},$$
The last equality  holds because   $\widetilde{h}_{i\bar{j}}(w_{0})=\delta_{i\bar j}$.
Notice that $X$ is global defined, i.e., $X(z,\zeta)= X(w,\widetilde{\zeta})$. That is $|\zeta|^{2} h(z)=|\widetilde{\zeta}|^{2} \widetilde{h}(w)$.
By the invariance of the first-order differential form, we have
$$d X(z,\zeta)= \sum_{i=1}^{m} X_{i}(z,\zeta)d\eta_{i}=dX(w,\widetilde{\zeta})=\sum_{i=1}^{m} X_{i}(w,\widetilde{\zeta})d\widetilde{\eta}_{i}=\widetilde{\zeta}d\widetilde{\zeta}.$$
Then, we get
$d\widetilde{\zeta}=\frac{1}{\widetilde{\zeta}}\sum_{i=1}^{m} X_{i}d\eta_{i}$.
This implies that
$$|d\widetilde{\zeta}|^{2}=\frac{1}{X(\widehat{z}_{0},\zeta)}\sum_{i,j=1}^{m} X_{i}(\widehat{z}_{0},\zeta)X_{j}(\widehat{z}_{0},\zeta)d\eta_{i}
d\overline{\eta}_{j}.$$

By the arbitrary  of $\hat{z}_{0}$ and Lemma \ref{cur 0},  at the $(z, \zeta)$, we have
\begin{eqnarray}\label{h sect 1}
R_{D}&=&-2\left(1+\frac{C}{Y^{m+1}}+\frac{\lambda_{0}}{mY}\right)\left\{\sum_{i,j=1}^{m}Y(\log X)_{i\bar j}d\eta_{i}d\overline{\eta}_{j}\right\}^{2}+YR_{ M}\\
&&-4\left(1-\frac{mC}{Y^{m+1}}\right)\left\{\sum_{i,j=1}^{m}Y(\log X)_{i\bar j}d\eta_{i}d\overline{\eta}_{j}\right\}\left\{\sum_{i,j=1}^{m}\frac{Y'}{X}X_{i}X_{\bar j}d\eta_{i}d\overline{\eta}_{j}\right\}\\
&&-\left(2+\frac{m(m-1)C}{Y^{m+1}}\right)\left\{\sum_{i,j=1}^{m}\frac{Y'}{X}X_{i}X_{\bar j}d\eta_{i}d\overline{\eta}_{j}\right\}^{2}.
 \end{eqnarray}

Moreover, we have
\begin{eqnarray*}
R_{D}
&=&-2\left(\frac{(m+1)C}{Y^{m+1}}+\frac{\lambda_{0}}{ mY}\right)\left\{\sum_{i,j=1}^{m}Y(\log X)_{i\bar j}d\eta_{i}d\overline{\eta}_{j}\right\}^{2}+YR_{ M}\\
&&-2\left(1-\frac{mC}{Y^{m+1}}\right)\left\{\sum_{i,j=1}^{m}\widetilde{g}_{i\bar j}d\eta_{i}d\overline{\eta}_{j}\right\}^{2}\\
&&-\left(\frac{m(m+1)C}{Y^{m+1}}\right)\left\{\sum_{i,j=1}^{m}\frac{Y'}{X}X_{i}X_{\bar j}d\eta_{i}d\overline{\eta}_{j}\right\}^{2}.
 \end{eqnarray*}
We obtain the conclusion.
\end{proof}

\begin{theorem}
Let $(L, h)$ be a negative line bundle over $M$ such that the dual bundle $(L^{*}, h^{-1})$ induces a complete K\"ahler-Einstein metric $g_{M}$
which Ricci curvature  $\lambda_{0}<1$.
Let $\widetilde{g}_{D}$ be the complete K\"ahler-Einstein metric of the disk bundle \eqref{D1} with Ricci curvature $-(m+1)$. Then $(D, \widetilde{g}_{D})$ is negatively holomorphically pinched if and only  if $(M, g_{M})$ is negatively holomorphically pinched and $\lambda_{0}<-m+2$,
where $m$ is the dimension of the disk bundle.
\end{theorem}
\begin{proof}
By Lemma \ref{lem:totals}, $(M, \alpha g_{M})$ is a totally geodesic submanifold of $(D(L), \widetilde{g}_{D})$, where $\alpha=\frac{1-\lambda_{0}}{m+1}$.
By Gauss equation, the holomorphic sectional curvature $H(\alpha g_{M}, Z)=H(g_{D}, Z)$ for any non-zero vector $Z\in T^{(1,0)}_{p}M$.
Suppose that $(D, \widetilde{g}_{D})$ is negatively holomorphically pinched, then  $(M, \alpha g_{M})$ is negatively holomorphically pinched.
This implies that $(M,  g_{M})$ is also  negatively holomorphically pinched.
Moreover, at the point $p$ in the zero section of the disk bundle $D(L)$,
we choose   the  holomorphic tangent  vector $\xi_{0}=\frac{\partial}{\partial \zeta}\in T_{p}D$. By Lemma \ref{cur 0},   the holomorphic sectional curvature
$
H(g_D, p, \xi_{0})=\frac{2(m+1)(\lambda_{0}+m-2)}{1-\lambda_{0}}.
$
If $\lambda_{0}\in [ -m+2, 1)$, then  $H(g_D, p, \xi_{0})\geq 0$.
It is a conflict with the assumption that $\widetilde{g}_{D}$ is negatively holomorphically pinched. The necessity  follows from it.

We now prove the sufficiency. Suppose that $ 2C_{1}\leq H(g_{M})\leq 2C_{2}<0$, i.e.,
 $$C_{1} (\sum_{i,j=1}^{n}(\log X)_{i\bar j})^{2}\leq R_{ M}\leq C_{2} (\sum_{i,j=1}^{n}(\log X)_{i\bar j})^{2}.$$
By Lemma \ref{cur 0},  we have
\begin{eqnarray}\label{equ:Ra}
R_{D}(z_{0}, \zeta )
&\leq&-2\left((1+\frac{C}{Y^{m+1}})+\frac{1}{mY}(\lambda_{0}-\frac{mC_{2}}{2})\right)\left\{\sum_{i,j=1}^{n}Y(\log X)_{i\bar j}dz_{i}\overline{dz_{j}}\right\}^{2}\\
&&-4\left(1-\frac{mC}{Y^{m+1}}\right)\left\{\sum_{i,j=1}^{n}Y(\log X)_{i\bar j}dz_{i}\overline{dz_{j}}\right\}\left\{Y'd\zeta d\overline{\zeta }\right\}\\
&&-\left(2+\frac{m(m-1)C}{Y^{m+1}}\right)\left\{Y'd\zeta d\overline{\zeta }\right\}^{2}.
 \end{eqnarray}
If $\lambda_{0}\in (-\infty, -m+2)$, then  the   coefficients of the second and the third tensors on the right hand side of \eqref{equ:Ra} are bounded above by  certain  negative constants by \eqref{inequ2} and \eqref{inequ3}.
For the first one, we divide two cases.
\begin{enumerate}
  \item   $ \lambda_{0}\in [-m, -m+2)$. Then we know $C<0$.
  Define
  $$f(t)=1+\frac{C}{t^{m+1}}+\frac{1}{mt}(\lambda_{0}-\frac{mC_{2}}{2})
  \ \ \ \text{on} \ \ \   (0,+\infty).$$
If $C_{2}\geq \frac{2\lambda_{0}}{m}$, then
 $f'(t)=(m+1)\frac{-C}{t^{m+2}}-\frac{1}{mt^{2}}(\lambda_{0}-\frac{mC_{2}}{2})>0.$
It is strictly increasing  on the open interval $ (0,+\infty)$.
Hence, we have
$$f(t)\geq f(\frac{1-\lambda_{0}}{m+1})=-\frac{(m+1)C_{2}}{2(1-\lambda_{0})}>0 \ \ \text{on} \ \  [\frac{1-\lambda_{0}}{m+1},+\infty).$$
If $C_{2}< \frac{2\lambda_{0}}{m}$, then we have $$f(t)\geq1+\frac{C}{t^{m+1}}\geq\frac{-(m+1)\lambda_{0}}{m(1-\lambda_{0})}>0 \ \  \text{on} \ \ [\frac{1-\lambda_{0}}{m+1},+\infty).$$ Since $Y(X)\in [\frac{1-\lambda_{0}}{m+1},+\infty)$,  we know
 the  coefficient of the first tensor on the right hand side of \eqref{equ:Ra} is bounded above by a certain  negative constant.

According to  \eqref{inequ1}-\eqref{inequ3}, we
 define a positive constant as follows:
 $$A=\min \left\{-\frac{(m+1)C_{2}}{2(1-\lambda_{0})}, -\frac{(m+1)\lambda_{0}}{m(1-\lambda_{0})}, -\frac{(m+1)(\lambda_{0}+m-2)}{2(1 -\lambda_{0})}\right\}.$$
Then we have
 \begin{eqnarray*}
R_{D}(z_{0}, \zeta )
&\leq&-2A\left(\left\{\sum_{i,j=1}^{n}Y(\log X)_{i\bar j}dz_{i}\overline{dz_{j}}\right\}^{2}+\left\{Y'd\zeta d\overline{\zeta }\right\}^{2}\right)\\
&\leq&-A\left(\sum_{i,j=1}^{n}Y(\log X)_{i\bar j}dz_{i}\overline{dz_{j}}+Y'd\zeta d\overline{\zeta }\right)^{2}=-A\widetilde{g}_{D}^{2}.
 \end{eqnarray*}
  \item  $\lambda_{0} \in   (-\infty, -m)$. This implies that $C>0$.
 Notice that
 $$1+\frac{C}{Y^{m+1}}+\frac{1}{mY}(\lambda_{0}-\frac{mC_{2}}{2})>1+\frac{\lambda_{0}}{mY}>1+\frac{\lambda_{0}}{m}\frac{m+1}{1-\lambda_{0}}
 =\frac{m+1}{m(1-\lambda_{0})}>0.$$
The coefficients of the  first  tensor   on the right hand  of \eqref{equ:Ra} is bounded above by a negative constant. Define $A=\min\left\{ \frac{m+1}{m(1-\lambda_{0})},1\right\}.$
 By the similar method of Item (1), we have $R_{D}(z_{0}, \zeta )\leq-A\widetilde{g}_{D}^{2}$.
\end{enumerate}

The discussion above shows that the holomorphic sectional curvature of $H(\widetilde{g}_{D})$ are bounded above by a negative constant.
As for the existence of a negative lower bound  for  $H(\widetilde{g}_{D})$, it can be obtained in a completely analogous
but more easier manner.  Hence, we omit the detains.
\end{proof}

\subsection{Sectional curvature}
We get an estimation  of the sectional curvature of the complete K\"ahler-Einstein metric  by the formula \eqref{Relation2}.
\begin{theorem}
Let $(L, h)$ be a negative line bundle over $M$ such that the dual bundle $(L^{*}, h^{-1})$ induces a complete K\"ahler-Einstein metric $g_{M}$
with Ricci curvature  $\lambda_{0}<1$.
Let $\widetilde{g}_{D}$ be the complete K\"ahler-Einstein metric of the disk bundle \eqref{D1} with Ricci curvature $-(m+1)$.
 If  $(M, g_{M})$ is a negatively
  pinched K\"ahler-Einstein manifold with  Ricci curvature  $\lambda_{0}<-m$ and the  sectional curvature are bounded above by $\frac{2\lambda_{0}}{m}$, then $(D, \widetilde{g}_{D})$ is negatively  pinched,
where $m$ is the dimension of the disk bundle.
\end{theorem}
\begin{proof}
For convenience,  we define three tensors as follows:
$$\widehat{g}_{1}=\sum_{i,j=1}^{m}Y(\log X)_{i\bar j}d\eta_{i}d\overline{\eta}_{j}, \
\widehat{g}_{2}=\sum_{i,j=1}^{m}\widetilde{g}_{i\bar j}d\eta_{i}d\overline{\eta}_{j}, \
 \widehat{g}_{3}=\sum_{i,j=1}^{m}\frac{Y'}{X}X_{i}X_{\bar j}d\eta_{i}d\overline{\eta}_{j}.$$
 Then $\widehat{g}_{1}=Yg_{M}$, $\widehat{g}_{2}=\widetilde{g}_{D}$.
 By \eqref{G}, two degenerate metrics $\widehat{g}_{1}$, $\widehat{g}_{3}$ are all dominated by the metric $\widehat{g}_{2}$  on $D$ and we have
\begin{eqnarray*}\label{eq:1}
R_{D}
&=&-2\left(\frac{(m+1)C}{Y^{m+1}}+\frac{\lambda_{0}}{ mY}\right)\widehat{g}_{1}^{2}+YR_{ M}-2\left(1-\frac{mC}{Y^{m+1}}\right)\widehat{g}_{2}^{2}
-\frac{m(m+1)C}{Y^{m+1}}\widehat{g}_{3}^{2}.
 \end{eqnarray*}
Let $g^{r}_{\alpha}$ be the corresponding Riemannian metrics for  $\widehat{g}_{\alpha}$.
If we take any two vectors $\xi=\frac{1}{2}(U-\sqrt{-1}JU), \eta=\frac{1}{2}( V-\sqrt{-1}J V)\in T^{(1,0)}D$,
i.e., $U=\xi+\overline{\xi},  V=\eta+\overline{\eta}$, then $g^{r}_{\alpha}(U, V)=\widehat{g}_{\alpha}(\xi,\eta)$.

By  Lemma \ref{lem:tauR} and Lemma \ref{lem:taug}, the Riemannian sectional tensor is
\begin{eqnarray*}
-R_{D}^{r}(U, V,U, V)&=&\tau (R_{D})(U, V,U, V)\\
&=&-2\left(\frac{(m+1)C}{Y^{m+1}}+\frac{\lambda_{0}}{mY}\right)\tau(\widehat{g}_{1}\otimes\widehat{g}_{1})(U, V,U, V)
+Y(-R_{ M}^{r}(U, V,U, V))\\
&&-2\left(1-\frac{mC}{Y^{m+1}}\right)\tau(\widehat{g}_{2}\otimes\widehat{g}_{2})(U, V,U, V)\\
&&-\frac{m(m+1)C}{Y^{m+1}}\tau(\widehat{g}_{3}\otimes\widehat{g}_{3})(U, V,U, V),
\end{eqnarray*}
where the tensor
$\tau(\widehat{g}_{\alpha}\otimes\widehat{g}_{\alpha})(U, V,U, V)
=\frac{1}{4}\left(||U\wedge  V||^{2}_{g^{r}_{\alpha}}+3g^{r}_{\alpha}(U,J V)^{2}\right)$.
Suppose that  $(M, g_{M})$ is a negatively
  pinched K\"ahler-Einstein manifold, i.e., there exist two negative constants such that
$$c_{1}\leq-\frac{R_{ D}^{r}(U, V,U, V)}{\|U\wedge  V\|^{2}_{g^{r}_{M}}}\leq c_{2}<0.$$ Since $\widehat{g}_{1}=Yg_{M}$, we have
 $$-R_{M}^{r}(U, V,U, V)\leq c_{2}\|U\wedge  V\|^{2}_{g^{r}_{M}}
=\frac{c_{2}}{Y^{2}}\|U\wedge  V\|^{2}_{g^{r}_{1}}
\leq \frac{c_{2}}{4Y^{2}}(||U\wedge  V||^{2}_{g^{r}_{1}}+3g^{r}_{1}(U,J V)^{2}).$$

Hence, we have
\begin{eqnarray*}
-R_{D}^{r}(U, V,U, V)
&\leq&-\frac{1}{2}\left(\frac{(m+1)C}{Y^{m+1}}+\frac{1}{Y}(\frac{\lambda_{0}}{m}-\frac{c_{2}}{2})\right)
\left(||U\wedge  V||^{2}_{g^{r}_{1}}+3g^{r}_{1}(U,J V)^{2}\right)
\\
&&-\frac{1}{2}\left(1-\frac{mC}{Y^{m+1}}\right)
\left(||U\wedge  V||^{2}_{g^{r}_{2}}
+3g^{r}_{2}(U,J V)^{2}\right)\\
&&-\frac{m(m+1)C}{4Y^{m+1}}\left(||U\wedge  V||^{2}_{g^{r}_{3}}+3g^{r}_{3}(U,J V)^{2}\right).
\end{eqnarray*}
If $c_{2}\leq\frac{2\lambda_{0}}{m}$ and $\lambda_{0}\leq-m$, the equations \eqref{inequ1} and  \eqref{inequ2} shows that the coefficients of the  three  tensors    on the right hand  are uniformly bounded from above and below by negative constants.
The second tensor will give a strictly negative contribution to the Riemannian sectional curvature, the remaining two tensors will give non-positive
contributions.
This implies that the  sectional curvature  is bounded above by a negative constant.
The existence of a negative lower bound  for  the  sectional curvature  can be obtained by the same way.
\end{proof}

\section{A classification of the spherical type  sphere bundles }\label{5}
\subsection{Disk bundle}
Let  $\pi:(L, h)\rightarrow M$ be a negative Hermitian line bundle over an $n$-dimensional complex manifold $M$.
Assume  that the curvature of the dual bundle $(L^*, h^{-1})$ induces a K\"ahler metric $g_{M}$ on $M$, whose the K\"ahler form is given by $\omega_{M}=\frac{\sqrt{-1}}{2}\partial \bar\partial\log h$.
 The disk bundle is defined by
 \begin{equation}
 D(L) := \{\zeta  \in L : |\zeta|_{h}^{2} < 1\},
 \end{equation}
where $|\zeta|_{h}$ is the norm of $\zeta$ with respect to the metric $h$.
The following $(1,1)$-form
\begin{equation*}\label{omega D0}
\omega_{D}:=\pi^{*}(\omega_M)-\frac{\sqrt{-1}}{2}\partial\bar\partial \log(1-|\zeta|_{h}^{2})
\end{equation*}
is positive on $D(L^{*})$.
Let $g_{D}$ be the  corresponding metric.
Its curvature tensor was given by us in \cite{Hao2026}.
 Denote by $Ric_{M}$ and $s_{M}$   the Ricci tensor and scalar curvature of $g_{M}$ respectively.


\begin{lemma}\cite{Hao2026}\label{ric}
The  Ricci tensor of $g_{D}$ is
$$Ric_{D}=-(n+2)g_{D}+(n+1)g_{M}+Ric_{M}.$$
\end{lemma}

\begin{lemma}\cite{Hao2025}
The  scalar curvature of $g_{D}$ is
$$s_{D}=-((n+1)n+s_{M})(|\zeta|_{h}^{2}-1)-(n+2)(n+1).$$
\end{lemma}

\begin{lemma}\cite{Hao2026}\label{h sect}
For any  fixed point $p_{0}\in D(L)$,  there exists a local coordinate system $(z, \zeta)$ around it
such that
$$g_{M}|_{p'_{0}}=\sum^{n}_{j,k=1} \delta_{j\bar k}dz_{j}\otimes\overline{dz}_{k}, \ \
g_{D}|_{p_{0}}=\frac{1}{1-|\zeta|^{2}}\left(\frac{d\zeta \otimes\overline{d\zeta}}{1-|\zeta|^{2}}+g_{M}|_{p'_{0}}\right),$$
where $p'_{0}$ is the project point of $p_{0}$ under the dual mapping $\pi^{*}$ from $L$ to $M$. At the point $p_{0}$,
the  curvature tensor of $g_{D}$  is
$$R_{D}=-2\left((g_{D})^{2}|_{p_{0}}-\frac{1}{1-|\zeta|^{2}}\left((g_{M})^{2}|_{p'_{0}}+\frac{1}{2}R_{M}
\right)\right).$$
In particular, the  holomorphic sectional curvature $ H(g_{D},p_{0})=-4$  if and only if $ H(g_{M},p'_{0})=-4$.
\end{lemma}
The Bochner tensor for a K\"ahler manifold is defined as follows:
\begin{eqnarray*}
                           B (\alpha ,\beta,\mu,\nu)
&=&R (\alpha ,\beta,\mu,\nu)-\frac{1}{n+2}\left[Ric (\alpha ,\mu)g (\beta,\nu)\right.\\
&&+Ric (\beta,\mu)g (\alpha ,\nu)+\left.Ric (\alpha ,\nu)g (\beta,\mu)+Ric (\beta,\nu)g (\alpha ,\mu)\right]\\
&&+\frac{S }{(n+1)(n+2)}[g (\alpha ,\mu)g (\beta,\nu)+g (\beta,\mu)g(\alpha ,\nu)],
\end{eqnarray*}
where  $R$ is the curvature tensor, $Ric$ is the Ricci curvature tensor, $S$ is the scalar curvature.
A Bochner-K\"ahler  manifold is a K\"ahler manifold with vanishing Bochner tensor.
Substituting the above curvature tensors into the formula for Bochner tensor, we obtain the following result.
\begin{lemma}\label{lem:Boch}
The  Bochner tensor of $g_{D}$ is vanish if and only if
$ H(g_{D})=-4$.
\end{lemma}
\begin{proof}
Sufficiency is trivial. We now prove necessity.
For any  fixed point $p_{0}\in D(L)$,  we take the local coordinate system $(z, \zeta)$ given by Lemma \ref{h sect}.
At the point $p_{0}$,
the Bochner tensor of $g_{D}$  for $(\alpha, \bar{\alpha}, \alpha, \bar{\alpha})\in T_{(1,0)}D\times T_{(0,1)}D\times T_{(1,0)}D\times T_{(0,1)}D$ is  as follows:
\begin{eqnarray*}
                           B (\alpha, \bar{\alpha}, \alpha, \bar{\alpha})
&=&-2\left((g_{D}(\alpha,\bar{\alpha}))^{2}-\frac{1}{1-|\zeta|^{2}}\left((g_{M}(\alpha, \bar{\alpha}))^{2}+\frac{1}{2}R_{M}(\alpha,\bar{\alpha}, \alpha,\bar{\alpha})
\right)\right)\\
&&-\frac{4}{n+3}\left(-(n+2)g_{D}(\alpha, \bar{\alpha})+(n+1)g_{M}(\alpha,\bar{\alpha})+Ric_{M}(\alpha,\bar{\alpha})\right) g_{D}(\alpha,\bar{\alpha})\\
&&-\frac{2 }{(n+2)(n+3)}\left(((n+1)n+s_{M})(|v|_{h}^{2}-1)+(n+2)(n+1)\right)(g_{D}(\alpha,\bar{\alpha}))^{2}\\
&=&\frac{2}{1-|\zeta|^{2}}\left((g_{M}(\alpha, \bar{\alpha}))^{2}+\frac{1}{2}R_{M}(\alpha, \bar{\alpha}, \alpha,\bar{\alpha})
\right)\\
&&-\frac{4}{n+3}\left((n+1)g_{M}(\alpha, \bar{\alpha})+Ric_{M}(\alpha, \bar{\alpha})\right) g_{D}(\alpha, \bar{\alpha})\\
&&-\frac{2}{(n+2)(n+3)}((n+1)n+s_{M})(|v|^{2}_{h}-1)(g_{D}(\alpha, \bar{\alpha}))^{2}.
\end{eqnarray*}
If the Bochner tensor of $g_{D}$ is vanish, then the coefficient of the tensor $|d\zeta|^{4}$ must vanish.
From
the formula of $g_{D}$ at $p_{0}$ in the local coordinate given by Lemma \ref{h sect},  we can get
$$\frac{2((n+1)n+s_{M})}{(n+2)(n+3)(1-|\zeta|^{2})^{3}}=0.$$
This implies that the scalar curvature $s_{M}=-(n+1)n$.
By the same way, the coefficient of the tensor $|d\zeta|^{2}$  also vanishes, i.e., $(n+1)g_{M}(\alpha, \bar{\alpha})+Ric_{M}(\alpha, \bar{\alpha})=0$.
This implies that the Ricci curvature of $g_{M}$ is  $-(n+1)$.
Finally, we get $$(g_{M}(\alpha, \bar{\alpha}))^{2}+\frac{1}{2}R_{M}(\alpha, \bar{\alpha}, \alpha,\bar{\alpha})=0.$$
This implies that the holomorphic sectional curvature  $ H(g_{M},p'_{0})=-4$.
By using Lemma \ref{h sect} again, we have $ H(g_{D},p_{0})=-4$.
\end{proof}
In \cite{Hao2026},  we had proved that  $(D(L),  g_{D})$  is a complete  K\"ahler manifold if
the K\"ahler manifold $(M, g_{M})$  is complete.   This implies the following result.
\begin{corollary}
Let $(M, g_{M})$ be a complete K\"ahler manifold. Then  $(D(L), g_{D})$ is a simply-connected  Bochner-K\"ahler manfold if and only if it is the complex hyperbolic space.
\end{corollary}

Define $S(L)=\{\zeta \in L : \|\zeta\|_{h}^{2} = 1\}.$  It is called a circle bundle.
Recall that a CR hypersurface  is called spherical if, for every $p$ in it, there is a neighborhood of $p$ in the CR hypersurface that is CR- diffeomorphic to an open piece of the sphere.
The following lemma is characterization of the spherical type  circle bundles.
\begin{lemma}\cite{Ebenfelt2023}\label{lem:E}
Let $(M, g)$ be a K\"ahler manifold of complex dimension $n$. Assume that  $(M, g)$ has constant Ricci eigenvalues. Then $S(L)$ is spherical if and only if
 $(M, g)$ is locally holomorphically isometric to one of the following:
\begin{enumerate}
\item $(\mathbb{B}_{n}, \lambda_{0} g_{-1})$ for some $\lambda_{0}\in R^{+}$,
\item $(\mathbb{CP}^{n}, \lambda_{0} g_{1})$ for some $\lambda_{0}\in R^{+}$,
\item $(\mathbb{C}_{n}, g_{0})$,
\item $(\mathbb{B}^{l} \times \mathbb{CP}^{n-l}, \lambda_{0} g_{-1}\times \lambda_{0} g_{1})$ for some $1 \leq l \leq n - 1$ and some $\lambda_{0}\in R^{+}$.
\end{enumerate}
where $g_{-\alpha}$ denotes the standard metric of the complex space form with Ricci curvature $-\alpha$.
\end{lemma}

Recall a result given by Huang and Li \cite{HuangL2020}  that: for any pseudoconvex domain with a strictly pseudoconvex boundary point, if  the Bergman metric  is K\"ahler-Einstein,  then the boundary near the point is locally spherical.
Combining this conclusion with Lemma \ref{lem:E} and their Proposition 1.9 in \cite{Ebenfelt2023}, we directly observe the following rigidity phenomenon. Because we merely put together their results, there is no doubt that this result belongs to them.

\begin{corollary}\cite{Ebenfelt2023,HuangL2020}
Let $\Omega$ be a simply-connected bounded pseudoconvex  domain in $\mathbb{C}^{n}$.
 Let $$D(\Omega)=
\left\{(z,\zeta)\in \Omega\times \mathbb{C}:|\zeta|^{2}<h^{-1}(z)\right\}
$$ be a bounded pseudoconvex  Hartogs domain over  $\Omega$.
If  $\sqrt{-1}\partial \overline{\partial}\log h$  induces  a complete K\"ahler metric with constant Ricci eigenvalues,
then the Bergman metric of $D(\Omega)$ is K\"ahler-Einstein if and only if $D(\Omega)$  is biholomorphically equivalent to
$\mathbb{B}^{n+1}$.
\end{corollary}
\begin{proof} It suffices  to prove the necessity. Huang and Li's result shows that
the strictly pseudoconvex boundary $\partial_{s}(D(\Omega)):=\{(z,\zeta)\in \Omega\times \mathbb{C}:|\zeta|^{2}=h^{-1}(z)\}$ is locally spherical.
Denote by $g_{\Omega}$  the complete  K\"ahler metric with K\"ahler form $\sqrt{-1}\partial \overline{\partial}\log h$.
Assume it has constant Ricci eigenvalues.
Then
$(\Omega, g_{\Omega})$ is locally holomorphically isometric to $(N, g_{N})$ for $n\geq2$,
where  $(N, g_{N})$ is one of the K\"aher manifolds (1)-(4) in Lemma \ref{lem:E}.
If $(N, g_{N})= (\mathbb{B}^{l}\times\mathbb{CP}^{n-l}, \lambda_{0} g_{-1}\times\lambda_{0} g_{1})$, then
  for any nonempty open subset $U\subset \Omega$, there exists a holomorphic mapping $\Phi$ such that   $g_{\Omega}|_{U}=\lambda_{0}\Phi^{*}(g_{-1}\oplus g_1)$.
  Shrinking $U$ sufficiently, $\Phi$ becomes injective on $U$ and maps $U$ onto a neighbourhood $V=\Phi(U)$. Clearly $\Phi^{-1}:V\subset\mathbb{B}^n\times\mathbb{CP}^n\to \Omega$ is also a local holomorphic isometry.
Since $\mathbb{B}^n\times\mathbb{CP}^n$ is simply connected and complete, and $(\Omega, g_{\Omega})$ is assumed complete, by de Rham's result and Lemma 1 in \cite{Lu1966}, $\Phi^{-1}$ extends to a holomorphic map $\widetilde \Phi^{-1}:\mathbb{B}^n\times\mathbb{CP}^n \to \Omega$.
In particular, $\mathbb{CP}^n$ is a compact submanifold of the pseudoconvex domain $\Omega$.  It is  impossible.
Lemma \ref{lem:E} implies the simply connected complete K\"ahler manifold $(\Omega, g_{\Omega})$ is one of the complex space forms.
Since the domain  $\Omega$ is  bounded and pseudoconvex,  $(\Omega, g_{\Omega})$ must be $(\mathbb{B}_{n}, \lambda_{0} g_{-1})$.
Moreover, the biholomorphic map between the bases can induce a biholomorphic map between the Hartogs domains.
Thus we know $B(\Omega, k)$ is biholomorphic to the domain
$\{(z, \zeta)\in\mathbb{C}^{n+1}: |\zeta|^{2p}+||z||^{2}<1\},$ where the number $p=\frac{\lambda_{0}}{n+1}>0$.
After that, by Proposition 1.9 in \cite{Ebenfelt2023}, the parameter
$p=1$ and $\lambda_{0}=-(n+1)$. Finally, we know $ B(\Omega, k)\cong \mathbb{B}^{n+k}$.
\end{proof}

\subsection{A ball bundle}
Let  $\pi:(L, h)\rightarrow M$ be a negative Hermitian line bundle over an $n$-dimensional complex manifold $M$, so that the dual bundle $(L^*, h^{-1})$ induces a K\"ahler metric $g_{M}$ on $M$, where the K\"ahler form is given by $\omega_{M}=-\frac{\sqrt{-1}}{2}\partial \bar\partial\log h^{-1}$.
For any fixed $k\in \mathbb{Z}^{+}$, set
$(E_{k}, H_{k}) = (L, h)\oplus \cdots\oplus (L, h).$
There are $k$ copies of  Hermitian line  bundle $(L, h)$ on the right hand side.
 The ball bundle is defined by
\begin{equation}\label{ball}
 B(E_{k}) := \{\zeta \in E_{k} : \|\zeta\|_{H_{k}}^{2} < 1\}.
\end{equation}
The boundary of $B(E_{k}) $ is given by
$ \partial  B(E_{k})=\partial M\cup  S(E_{k}),$
where $M$ is identified with the zero section of the vector bundle, and $S(E_{k})$ is the sphere bundle
\begin{equation}\label{Sph}
S(E_{k})=\{\zeta \in E_{k} : \|\zeta\|_{H_{k}}^{2} = 1\}.
\end{equation}
It is well known that any point in $S(E_{k})$ is
strictly pseudoconvex (see Proposition 5.3 in \cite{Coevering2012}).
Denote by $g_{M}$ the induced K\"ahler metric of $\omega_M$.
Define a $(1,1)$-form as follows:
\begin{equation}\label{omega Dk}
\omega_{B(E_{k})}:=\pi^{*}(\omega_M)-\frac{\sqrt{-1}}{2}\partial\bar\partial \log(1-\|\zeta\|_{H_{k}}^{2}) \ \text{on} \ E_{k}.
\end{equation}
By Lemma 3 in \cite{Hao2026}, it induces a complete K\"ahler metric on $B(E_{k})$ if $g_{M}$ is complete. We  denote it by $g_{B(E_{k})}$.
We can give a  classification of the spherical type  sphere bundles  \eqref{Sph}.

\begin{theorem}\label{thm s}
Let $(M, g_{M})$ be a K\"ahler manifold.
Let $ B(E_{k}) $ be the ball bundle defined by \eqref{ball}.
Assume that the rank $k\geq 2$. If  $S(E_{k})$  is spherical, then
 $(M, g_{M})$ is locally holomorphically isometric to
 $(\mathbb{B}_{n},  g_{hyp})$,
  where
  $g_{hyp}= -\sum_{i,j=1}^{n}\frac{\partial^{2}\log(1-||z||^{2})}{\partial z\bar\partial \bar z}dz_{i}\otimes \overline{dz}_{j}$.
\end{theorem}
\begin{proof}
Firstly, we will prove that the ball bundle $B(E_{j})$ can be seen as a unit disk bundle over $B(E_{j-1})$ for $1\leq j\leq k$.
Actually, $E_{k}$ is a line bundle over  $E_{k-1}$.
Restrict $E_{k}$  to $B(E_{k-1})$, and denote it by $\pi_{k}:L_{k}\rightarrow B(E_{k-1})$.
Since $E_{k}=E_{k-1}\oplus L$,  for $\zeta\in E_{k}$, we have
$\zeta=\zeta'\oplus \zeta_k$, where $\zeta'\in E_{k-1}$ and $\zeta_{k}$ is the  fiber with  dimension one.
Define $\widetilde{h}_{k}=h(1-\|\zeta'\|^{2}_{H_{k-1}})^{-1}$ as a metric on the line bundle $L_k$.
The curvature tensor for the metric $\widetilde{h}_{k}^{-1}$ on dual bundle $E^{*}_{k-1}$ is
$$\Theta=-\partial\overline{\partial}\log \widetilde{h}_{k}^{-1}=\partial \bar \partial \log h-\partial\bar\partial \log(1-\|\zeta'\|_{H_{k-1}}^{2}), $$ i.e.,
$\frac{\sqrt{-1}}{2}\Theta=\omega_{B(E_{k-1})}.$
It implies that the Hermitian line bundle $(L_{k}, \widetilde{h}_{k})$ over  $B(E_{k-1})$ is negative. Define a
 disk bundle:
 \begin{equation}\label{equ:DLk}
D(L_{k})=  \left\{\zeta_{k} \in L_{k}: |\zeta_{k}|_{\widetilde{h}_{k}}^{2} < 1, \zeta' \in B(E_{k-1})\right\},
 \end{equation}
where $B(E_{0})$ denotes the manifold $M$. Then we have
$D(L_{k})= B(E_{k})$.
The boundary of $D(L_{k})$ is given by
$$ \partial  D(L_{k})=\partial B(E_{k-1})\cup  S(L_{k}),$$
where
$S(L_{k})=\left\{\zeta_{k} \in L_{k}: |\zeta_{k}|_{\widetilde{h}_{k}}^{2} = 1, \zeta' \in B(E_{k-1})\right\}.$
 Since $\partial B(E_{k-1})=\partial M\cup S(E_{k-1})$, we know $$ \partial  D(L_{k})=\partial M\cup S(E_{k-1})\cup  S(L_{k}).$$
Notice that
 $\partial  D(L_{k})=\partial B(E_{k})=\partial M \cup  S(E_{k}).$ We get $S(E_{k})=S(E_{k-1})\cup  S(L_{k})$, i.e.,
 $$S(L_{k})=S(E_{k})\setminus S(E_{k-1}),$$ where  $S(E_{k-1})$ is $(2n+2k-3)$-dimensional real hypersurface.
Define  a
$(1,1)$-form as follows:
  $$\omega_{D(L_{k})}=\pi_{k}^{*}(\omega_{B(E_{k-1})})-\frac{\sqrt{-1}}{2}\partial\bar\partial \log(1-|\zeta_{k}|_{\widetilde{h}_{k}}^{2}),
$$
Then we have $\omega_{D(L_{k})}=\omega_{B(E_{k})}.$
Repeat this process, we can see that $B(E_{j})$ is a unit disk bundle over $B(E_{j-1})$ for $1\leq j\leq k$.
Moreover,  $\omega_{B(E_{j-1})}=\omega_{D(L_{j-1})}=\omega_{D(L_{j})}|_{M}=\omega_{B(E_{j})}|_{M}$.

Next, we prove the necessity.
If $S(E_{k})$ is locally spherical, then $S(L_{k})$ is locally spherical.
The work of Webster \cite{Webster1977} (see Proposition 4 in \cite{Wang2019}) shows that $S(L_{k})$ is spherical if and
only if $(D(L_{k-1}), \omega_{D(L_{k-1})})$ is a Bochner-K\"ahler manifold.
Then $(B(E_{k-1}), \omega_{B(E_{k-1})})$ is a Bochner-K\"ahler manifold.
By Lemma \ref{lem:Boch}, we know it is locally holomorphically isometric to $(\mathbb{B}_{n+k-1}, g_{hyp})$.
By repeatedly applying Lemma \ref{h sect}, we obtain $(B(E_{j}), \omega_{B(E_{j})})$ is locally holomorphically isometric to $(\mathbb{B}_{n+j}, g_{hyp})$ for $j=1,\cdots,k-1$,  and   $(M, g)$ is locally holomorphically isometric to $(\mathbb{B}_{n}, g_{hyp})$.
\end{proof}

\begin{theorem}\label{thm:4}
Let $\Omega$ be a simply-connected bounded pseudoconvex  domain in $\mathbb{C}^{n}$.
 Let
 $$B(\Omega, k)=\left\{(z,\zeta)\in \Omega\times \mathbb{C}^{k}:\|\zeta\|^{2}<h^{-1}(z)\right\}$$
  be a bounded pseudoconvex  Hartogs domain over  $\Omega$.
If  $\frac{\sqrt{-1}}{2}\partial \overline{\partial}\log h$  induces  a complete K\"ahler metric on $\Omega$ and the fiber dimension $k\geq 2$,
then the Bergman metric of $B(\Omega, k)$ is K\"ahler-Einstein if and only if $B(\Omega, k)$ is biholomorphically equivalent to
$\mathbb{B}^{n+k}$.
\end{theorem}
\begin{proof} We only need to prove the necessity, since the sufficiency is obvious.
According to the discussion of the unit ball bundle \eqref{ball},  any point in the boundary
$\partial_{s}(B(\Omega, k))=\{(z,\zeta)\in \Omega\times \mathbb{C}^{k}:\|\zeta\|^{2}=h^{-1}(z)\}$ is strictly pseudoconvex.
If the Bergman metric of $B(\Omega, k)$ is K\"ahler-Einstein, by Huang and Li'result in \cite{HuangL2020}, the boundary $\partial_{s}(B(\Omega, k))$ is locally spherical.

Let $g_{\Omega}$  be the complete K\"ahler metric with K\"ahler form $\omega_{\Omega}=\frac{\sqrt{-1}}{2}\partial \overline{\partial}\log h$.
For any $j\in \mathbb{Z}^{+}$,
the Hartogs domain $B(\Omega, j)=\{(z,\zeta)\in \Omega\times \mathbb{C}^{j}:\sum_{i=1}^{j}|\zeta_{i}|^{2}h(z)<1\}$ can be  seen as a unit ball bundle in the Hermitian vector bundle $(\Omega\times \mathbb{C}^{j}, h\oplus\cdots\oplus h)$ over $(\Omega, g_{\Omega})$.
Let $g_{j}$   be the complete K\"ahler metric with K\"ahler form
$$\omega_{j}=\omega_{\Omega}-\frac{\sqrt{-1}}{2}\partial\bar\partial \log(1-\sum_{i=1}^{j}|\zeta_{i}|^{2}h(z)).$$
We rewrite the domain  $B(\Omega, j)$ for $j\geq2$ as follows:
$$
 B(\Omega, j)=\left\{(z,\zeta_{j})\in B(\Omega, j-1) \times \mathbb{C}:|\zeta_{j}|^{2}h(z)(1-\sum_{i=1}^{j-1}|\zeta_{i}|^{2}h(z))^{-1}<1\right\},
$$
and
$
 B(\Omega, 1)=\left\{(z,\zeta_{1})\in \Omega \times \mathbb{C}:|\zeta_{1}|^{2}h(z)<1\right\}.
$
The Hartogs domain $B(\Omega, j)$ for $j\geq2$ can be  seen as a unit disk bundle in the Hermitian line bundle $(B(\Omega, j-1)\times \mathbb{C}, h(z)(1-\sum_{i=1}^{j-1}|\zeta_{i}|^{2}h(z))^{-1})$ over $(B(\Omega, j-1), g_{j-1})$
and $B(\Omega, 1)$ is a  unit disk bundle over $(\Omega, g_{\Omega})$.
By Theorem \ref{thm s}, there exists a biholomorphic mapping $z^{*}=F(z)$ from $\Omega$ to $\mathbb{B}^{n}$ such that $g_{\Omega}=F^{*}g_{hyp}$, i.e.,
$$\partial \bar\partial\log h(z)=-\partial \overline{\partial}\log (1-||F(z)||^{2}) \ \text{on} \ \Omega.$$
Let $z=F^{-1}(z^{*})$ be the inverse mapping of $F$, then  $(F^{-1})^{*} g_{\Omega}=g_{hyp}$ and
$$\partial \bar\partial\log h(F^{-1}(z^{*}))=-\partial \overline{\partial}\log (1-||z^{*}||^{2}) \ \text{on} \ \mathbb{B}^{n+k}.$$
 Remove the operator $\partial \bar\partial$,  there is a holomorphic function $f$ on $\Omega$ and  a holomorphic function $f^{*}$ on $\mathbb{B}^{n}$  such that
 \begin{equation}\label{equ:h F}
 h^{-1}(z)|f(z)|^{2}= (1-||F(z)||^{2}) \
\text{and} \
 h^{-1}(F^{-1}(z^{*}))= (1-||z^{*}||^{2})|f^{*}(z^{*})|^{2}.
 \end{equation}

 Define a holomorphic map
 $\widetilde{F}(z, \zeta):  B(\Omega, k)\rightarrow \mathbb{B}^{n}\times\mathbb{C}^{k}$, $(z^{*}, \zeta^{*})=\widetilde{F}(z, \zeta)$, where
$z^{*}=\widetilde{F}_{1}(z, \zeta)=F(z)$ and $\zeta^{*}=\widetilde{F}_{2}(z,\zeta)=e^{i\theta_{1}}\zeta f(z)$.
 This implies that
 \begin{eqnarray*}
 \|z^{*}\|^{2}+\|\zeta^{*}\|^{2}&=&||F(z)||^{2}+\|\zeta\|^{2}|f(z)|^{2}\\
 &=&1-\left(h^{-1}(z)-\|\zeta\|^{2}\right)|f(z)|^{2}<1.
 \end{eqnarray*}
Hence, we have
  $\widetilde{F}( B(\Omega, k))\subseteq \mathbb{B}^{n+k}.$ Let $\omega_{hyp}$ be the K\"ahler form of $\mathbb{B}^{n+k}$. By \eqref{equ:h F}, we obtain
\begin{eqnarray*}
\widetilde{F}^{*}\omega_{hyp}&=&-\frac{\sqrt{-1}}{2}\partial \overline{\partial}\log\left(1-\|F(z)\|^{2}-\|e^{i\theta_{1}}\zeta f(z)\|^{2}\right)\\
&=&\omega_{\Omega}-\frac{\sqrt{-1}}{2}\partial\bar\partial \log\left(1-\|\zeta\|^{2}h(z)\right)\\
&=&\omega_{B(E_{k})}.
\end{eqnarray*}
We also can
 define another  holomorphic map $\widetilde{F}^{*}(z^{*}, \zeta^{*})$ form  $\mathbb{B}^{n+k}$ to  $\Omega\times\mathbb{C}^{k}$, such that $(z,\zeta)=\widetilde{F}^{*}(z^{*}, \zeta^{*})$, where
$z=\widetilde{F}^{*}_{1}(z^{*}, \zeta^{*})=F^{-1}(z^{*})$ and $\zeta=\widetilde{F}^{*}_{2}(z^{*},\zeta^{*})=e^{i\theta_{2}}\zeta^{*}(f^{*})(z^{*})$.
This implies that
 \begin{eqnarray*}
 \|\zeta\|^{2}h(z)&=&\|\zeta^{*}\|^{2} |f^{*}(z^{*})|^{2}h(F^{-1}(z^{*}))\\
 &=&\|\zeta^{*}\|^{2}(1-||z^{*}||^{2})^{-1}
 <1.
 \end{eqnarray*}
Hence, we have $\widetilde{F}^{*}(\mathbb{B}^{n+k} )\subseteq B(\Omega, k).$
Take $z^{*}=F(z)$ in the second equation of \eqref{equ:h F}, then we have
 $h^{-1}(z)=(1-\|F(z)\|^{2})|f^{*}(F(z))|^{2}$.
 Compare it with the first equation of \eqref{equ:h F}, we have
 $|f^{*}(F(z))|^{2}=|f(z)|^{-2}$.
 By the maximum of the modulus of a holomorphic function, we get $f^{*}(F(z))=e^{i\theta_{3}}f^{-1}(z)$.
For the map $\widetilde{F}^{*}$, we take $\theta_{2}=-\theta^{1}-\theta^{3}$.  Then we get $$\widetilde{F}^{*}\circ F(z,\zeta)=(z,e^{i(\theta_{1}+\theta_{2})}\zeta f^{*}(F(z))f(z))=(z,\zeta).$$ Thus  $\widetilde{F}^{*}$ is the inverse map of $\widetilde{F}$.
Finally, the biholomorphic isometric  map $F$  from $\Omega$ to  $\mathbb{B}^{n}$ induces  a biholomorphic map $\widetilde{F}$  from $B(\Omega, k)$ to  $\mathbb{B}^{n+k}$.
We complete the proof.
 \end{proof}

\begin{remark}
In Theorem \ref{thm:4}, if $h$  is the positive powers of the Bergman kernel function for $\Omega$,
then $g_{\Omega}$ is the Bergman metric up to  scaling.
Recall Lu's well-known uniformization theorem: a bounded domain in $\mathbb{C}^{n}$ with a complete Bergman metric of constant holomorphic sectional curvature is biholomorphic to the Euclidean ball \cite{Lu1966}. Thus $\Omega$ is biholomorphic equivalent to $\mathbb{B}^{n}$.
By the same discussion of Theorem \ref{thm:4},
we get $ B(\Omega, k)\cong \mathbb{B}^{n+k}$.
Hence, the simply-connectedness assumption for the domain $\Omega$   can be removed in this case.
\end{remark}
\begin{remark}
\label{rem:homo B}
The domain $B(\Omega, k)$ is homogeneous if and only if it is biholomorphic to the unit ball.
Actually, it can be proved by Wong-Rosay theorem directly.
Let $P$ be a  point in the strictly pseudoconvex boundary $\{(z,\zeta)\in \Omega\times \mathbb{C}^{k}:\|\zeta\|^{2}=h^{-1}(z)\}$.
Let $Q$, $P_{j}$ be some interior points  in $B(\Omega, k)$ such that $P_{j} \rightarrow P$ when $j \rightarrow \infty$.
Assume that $B(\Omega, k)$ is homogeneous. There exists a familly of $f_{j}\in \mathrm{Aut}(B(\Omega, k))$
such that $f_{j}(Q)=P_{j} \rightarrow P$ as $j \rightarrow \infty$.
Thus $P$ is a boundary orbit
accumulation point for the action of the non-compact group $\mathrm{Aut}(B(\Omega, k))$.
The result follows from Wong-Rosay theorem in \cite{Wong1977}\cite{Rosay1979}.
\end{remark}


\end{document}